\documentclass[oneside,english]{amsart}
\usepackage[T1]{fontenc}
\usepackage[latin9]{inputenc}
\usepackage{color}
\usepackage{amsbsy}
\usepackage{amstext}
\usepackage{amsthm}
\usepackage{amssymb}
\usepackage{graphicx}
\usepackage[all]{xy}
\PassOptionsToPackage{normalem}{ulem}
\usepackage{ulem}

\makeatletter
\numberwithin{equation}{section}
\numberwithin{figure}{section}
\theoremstyle{plain}
\newtheorem{thm}{\protect\theoremname}[section]
  \theoremstyle{plain}
  \newtheorem{lem}[thm]{\protect\lemmaname}
  \theoremstyle{plain}
  \newtheorem{prop}[thm]{\protect\propositionname}
  \theoremstyle{plain}
  \newtheorem{cor}[thm]{\protect\corollaryname}
  \theoremstyle{definition}
  \newtheorem{defn}[thm]{\protect\definitionname}
  \theoremstyle{remark}
  \newtheorem{rem}[thm]{\protect\remarkname}
  \theoremstyle{remark}
  \newtheorem{claim}[thm]{\protect\claimname}

\makeatother

\usepackage{babel}
  \providecommand{\claimname}{Claim}
  \providecommand{\corollaryname}{Corollary}
  \providecommand{\definitionname}{Definition}
  \providecommand{\lemmaname}{Lemma}
  \providecommand{\propositionname}{Proposition}
  \providecommand{\remarkname}{Remark}
\providecommand{\theoremname}{Theorem}

\begin{document}

\title{Schwartz Functions on Quasi-Nash Varieties}

\author{Boaz Elazar}

\address{Dept. of Mathematics, The Weizmann Institute of Science, Rehovot
76100, Israel}

\email{boaz.elazar@weizmann.ac.il}
\begin{abstract}
We introduce a new category called Quasi-Nash, unifying Nash manifolds
and algebraic varieties. We define Schwartz functions, tempered functions
and tempered distributions in this category. We show that properties
that hold on affine spaces, Nash manifolds and algebraic varieties,
also hold in this category.
\end{abstract}

\maketitle

\section{Introduction}

Schwartz functions are named after Laurent Schwartz, who defined them
in $\mathbb{R}^{n}$. In $\mathbb{R}^{n}$ they are usually defined
as smooth functions which decay to zero with all their derivatives
faster than the inverse of any polynomial when reaching infinity.
We say $f$ is a Schwartz function on $\mathbb{R}$, for example,
if for any $n,k\in\mathbb{N}\cup\left\{ 0\right\} $ we have $|x^{n}f^{\left(k\right)}|<\infty$
where $f^{\left(k\right)}$ is the $k$'th derivative of $f$. The
space of all Schwartz functions on $\mathbb{R}^{n}$ will be denoted
by $\mathcal{S}\left(\mathbb{R}^{n}\right)$. Schwartz functions on
$\mathbb{R}^{n}$ have some nice properties such as: $\mathcal{S}\left(\mathbb{R}^{n}\right)$
is a Fréchet space, $\mathcal{S}\left(\mathbb{R}^{n}\right)$ is invariant
under Fourier transform and every function in $\mathcal{S}\left(\mathbb{R}^{n}\right)$
is integrable.

Later, in {[}dC, AG{]}, Schwartz functions were defined on Nash manifolds,
which are smooth semi-algebraic varieties. For a Nash manifold $M$,
$f$ is said to be Schwartz on $M$ if for any Nash differential operator
$D$ we have $||Df||_{\infty}<\infty$. Nash differential operator
on $M$ means an element of the algebra generated by multyplying by
Nash functions and by deriving along Nash sections of the tangent
bundle.

Lately, in {[}ES{]} we defined Schwartz functions on real algebraic
varieties which might have singularities, and showed how the affine
algebraic varieties share the properties of Schwartz functions on
Nash manifolds. We also showed that some of the results hold in the
general case. 

In this paper we define a category such that both the Nash manifolds
and the algebraic varieties are subcategories of the new one. Moreover,
this category enables us to prove the rest of the claims about Schwartz
functions on general algebraic varieties. We call this new category
Quasi-Nash, or QN, where its affine objects correspond to semi-algebraic
subsets of $\mathbb{R}^{n}$ and the morphisms are locally restrictions
of Nash maps. A general variety is defined as a glueing of open affine
varieties. This new category slightly extends the category of Nash
varieties.

The definitions of this category, and some usefull lemmas appear in
section 3.

The main results about Schwartz functions in this paper include:
\begin{lem}
Isomorphic QN varieties $X_{1}\cong X_{2}$, imply an isomorphism
of the Fréchet spaces $\mathcal{S}\left(X_{1}\right)\cong\mathcal{S}\left(X_{2}\right)$
where $\mathcal{S}\left(X_{i}\right)$ is the space of Schwartz functions
on $X_{i}$ (Lemma \ref{lem-First-Result}).
\end{lem}

The next Proposition deals with tempered functions. Informally, a
tempered function on $\mathbb{R}^{n}$ is a smooth function bounded
by a polynomial, and so does any of its derivatives. We define a tempered
function on a QN variety later on.
\begin{prop}
\label{prop-elor}(Tempered partition of unity) Let $\{V_{i}\}_{i=1}^{m}$
be a finite open cover of a QN variety $X$. Then, there exist tempered
functions $\{\beta_{i}\}_{i=1}^{m}$ on $X$, such that $supp(\beta_{i})\subset V_{i}$
and $\sum\limits _{i=1}^{m}\beta_{i}=1$ (Proposition \ref{prop-part-of-unity}).
\end{prop}

\begin{thm}
\label{thm-hardest}For a QN variety $X$ and a closed subset $Z\subset X$,
define $U:=X\setminus Z$ and $W_{Z}:=\{\phi\in\mathcal{S}(X)|\phi\text{ is flat on }Z\}$.
Then $W_{Z}$ is a closed subspace of $\mathcal{S}(X)$ (and so it
is a Fréchet space), and extension by zero $\mathcal{S}(U)\to W_{Z}$
is an isomorphism of Fréchet spaces, whose inverse is the restriction
of functions (Theorem \ref{thm-general-char}).
\end{thm}

It should be noted that the proof of Theorem \ref{thm-hardest} was
the hardest to prove. Usually, a function $f:\mathbb{R}^{n}\to\mathbb{R}$
is said to be flat at a point $p\in\mathbb{R}^{n}$ if $f$, and all
its derivatives, vanish at $p$. We had to make clear what it means
for a function on a singular variety to be flat on a singular point
$p$. For subvarieties, of $\mathbb{R}^{n}$ we defined it to be a
restriction of a smooth function on $\mathbb{R}^{n}$ which is flat
at $p$. There are several other approaches to do so (see {[}BMP2,
F{]} for example), but {[}BMP2{]} shows they are all equivalent in
our case. Then, Theorem \ref{thm-hardest} turns out to be a Whitney
type extension problem. We proved it using subanalytical geometry
results in {[}BM1, BM2, BMP1, BMP2{]}. Theorem \ref{thm-hardest}
also enables us to define Schwarz functions by a local condition rather
then by the global ones we used. Instead of demanding a function that
decays ``fast at infinity'', we just have to demand a smooth function
that is flat on the points ``added at infinity'' in some compactification
process.

Furthermore, Theorem \ref{thm-hardest} gives us one more important
result regarding tempered distributions. The space of tempered distributions
is the space of linear continuous functionals on $\mathcal{S}\left(X\right)$.
It is denoted by $\mathcal{S}^{*}\left(X\right)$. Theorem \ref{thm-hardest}
implies that for any open $U\subset X$, the restriction morphism
of tempered distributions $\mathcal{S}^{*}(X)\to\mathcal{S}^{*}(U)$
is onto. This is not the case for general distributions. E.g. take
the compactification of $\mathbb{R}$ into a circle. The distribution
$e^{x}dx$ on $\mathbb{R}$ cannot be extended to the circle.
\begin{cor}
\label{cor-eimim}Let $X$ be a QN variety. Then the assignment of
the space of Schwartz functions (respectively tempered functions,
tempered distributions) to any open $U\subset X$, together with the
extension by zero $Ext_{U}^{V}$ from $U$ to any other open $V\supset U$
(restriction of functions, restrictions of functionals from $\mathcal{S}^{*}(V)$
to $\mathcal{S}^{*}(U)$), form a flabby cosheaf (sheaf, flabby sheaf)
on $X$ (Corollaries \ref{cor-Sch-cosheaf}, \ref{cor-tempered-sheaf},
\ref{cor-Dist-sheaf}). 
\end{cor}

We would like to emphasize we proved Proposition \ref{prop-elor},
Theorem \ref{thm-hardest}, and Corollary \ref{cor-eimim} also for
non-affine varieties in this category, what we could not do in the
algebraic category.

\subsection*{Structure of this paper}

In \textbf{Section 2} we give preliminary definitions and results
we use in this paper. Most of them concern with Nash manifolds and
Schwartz functions on them, and basic properties of Fréchet spaces.\\
In \textbf{Section 3} we define the new category and show some nice
properties it has.\\
In \textbf{Section 4} we define Schwartz functions, tempered functions
and tempered distributions in this category, and prove some claims
about them.\\
In \textbf{Section 5} we show the co-sheaf structure of Schwartz functions,
and the sheaf structures of tempered functions and distributions.\\
In \textbf{Section 6} we define vector bundles over QN varieties and
show some properties that hold on those bundles.\\
Finally, in \textbf{Appendix A} we build some tools needed for tempered
partition of unity.

\subsection*{Conventions}

Throughout this paper, we use the restricted topology of semi algebraic
sets over $\mathbb{R}^{n}$, unless otherwise stated. For the definition
of restricted topology see \ref{AG-def-rest-top}.

We also use the convention that for two varieties of some kind $X\subset M$,
we will denote by $I_{Sch}^{M}\left(X\right)$ the ideal of Schwartz
functions on $M$ that vanish identically on $X$.

We say a function $f$ is smooth over $M$ if $f\in C^{\infty}\left(M\right)$.

\subsection*{Acknowledgments}

I would like to extend my gratitude to my supervisor Prof. Dmitry
Gourevitch, who taught me how to cross the barriers this work presented.
A great thanks goes to Prof. Avraham Aizenbud, who suggested different
approaches for defining our objects and dealing with the problems
we started with. I would like to thank Prof. William A. Casselman,
for his idea regarding the definition of Schwartz functions on closed
subsets as restrictions of Schwartz functions from their neighborhoods.
I thank Prof. Dmitry Novikov for patiently answering questions that
arose along the way. Finally, I would like to thank Prof. Alessandro
Tancredi for helping me understand the subtleties of the Nash varieties
category.

\section{Preliminaries}

We shall dedicate this section to definitions and results used in
this paper. They will include a definition of a semi-algebraic set,
and the algebraic Alexandrov compactification (2.1-2.3), Fréchet spaces
(2.5-2.7), Nash manifolds and Schwartz functions on Nash manifolds
(2.8-2.10), Schwartz and tempered functions over Nash manifolds (2.11-2.19).
\begin{defn}
{[}BCR{]} A semi-algebraic subset of $\mathbb{R}^{n}$ is a subset
of the form $\bigcup\limits _{i=1}^{n}\bigcap\limits _{j=1}^{m_{i}}\left\{ x\in\mathbb{R}^{n}|p_{i,j}\left(x\right)>0\text{ or }p_{ij}\left(x\right)=0\right\} $
where $p_{ij}\in\mathbb{R}[x_{1},...,x_{n}]$. 
\end{defn}

\begin{rem}
An affine algebraic variety is called \textit{complete} if \uline{any}
regular function on it is bounded (cf. {[}BCR 3.4.9 and 3.4.10{]}).
Thus, a closed embedding of a complete affine algebraic variety is
compact in the Euclidean topology on $\mathbb{R}^{n}$.
\end{rem}

\begin{prop}
\label{prop-Alexandrov}(Algebraic Alexandrov compactication {[}BCR,
Proposition 3.5.3{]}). Let $X$ be an affine algebraic variety that
is not complete, then there exists a pair $\left(\dot{X},i\right)$
such that:
\end{prop}

\begin{enumerate}
\item $\dot{X}$ is a complete affine algebraic variety.
\item $i:X\to\dot{X}$ is an algebraic isomorphism from $X$ onto $i\left(X\right)$.
\item $\dot{X}\backslash i\left(X\right)$ consists of a single point.
\end{enumerate}
The chain rule for deriving composite functions can be extended to
higher derivatives and higher dimensions. A relevant result is as
follows:
\begin{lem}
\label{lem-Faa-di}{[}CS, Theorem 2.1{]} Let $x_{0}\in\mathbb{R}^{d}$,
$V\subset\mathbb{R}^{d}$ be some open neighborhood of $x_{0}$ and
$g:V\to\mathbb{R}^{m},\:g\in C^{k}\left(V,\mathbb{R}^{m}\right)$,
for some $k\in\mathbb{N}$. Let $U\subset\mathbb{R}^{n}$ be some
open neighborhood of $g\left(x_{0}\right)$ and $f:U\to\mathbb{R},\:f\in C^{k}\left(U\right)$.
Assume $f$ is $k$-flat at $g(x_{0})$, i.e. its Taylor polynomial
of degree $k$ at $g(x_{0})$ is zero. Then $f\circ g:g^{-1}\left(U\right)\to\mathbb{R}$
is $k$-flat at $x_{0}$
\end{lem}

\textbf{Fréchet spaces}
\begin{prop}
\label{prop-Closed-Frechet-subspace}{[}T, Chapter 10{]}. A closed
subspace of a Fréchet space is a Fréchet space (in the induced topology).
\end{prop}

\begin{thm}
\label{thm-Banach-open-mapping}(Banach open mapping - {[}T, Chapter
17, Corollary 1{]}). A bijective continuous linear map from a Fréchet
space to another Fréchet space is an isomorphism.
\end{thm}

$\:$
\begin{thm}
\label{thm-Hahn-Banach}(Hahn-Banach - {[}T, Chapter 18{]}). Let $F$
be a Fréchet space, and $K\subset F$ a closed subspace. By Proposition
\ref{prop-Closed-Frechet-subspace} $K$ is a Fréchet space (with
the induced topology). Define $F^{*}$ (respectively $K^{*}$) to
be the space of continuous linear functionals on $F$ (on $K$). Then
the restriction map $F^{*}\to K^{*}$ is onto.
\end{thm}

$\:$

\textbf{Nash manifolds}
\begin{defn}
\label{AG-def-rest-top}A restricted topological space $M$ is a set
$M$ equipped with a family of subsets of $M$, including $M$ and
the empty set, called the set of open subsets of $M$, that is closed
with respect to \textbf{finite} unions and finite intersections.

Therefore, we will consider only finite open covers in restricted
topology.
\end{defn}

$\:$
\begin{defn}
An \textbf{$\mathbb{R}$-space} is a pair $\left(M,\mathcal{O}_{M}\right)$
where $M$ is a restricted topological space and $\mathcal{O}_{M}$
a sheaf of $\mathbb{R}$-algebras over $M$ which is a subsheaf of
the sheaf $C_{M}$ of all continuous real-valued functions on $M$.

A continuous map $\varphi:\left(M,\mathcal{O}_{M}\right)\to\left(N,\mathcal{O}_{N}\right)$
is called a morphism of $\mathbb{R}$-spaces if for any open subset
$U\subset N$ and any $f\in\mathcal{O}_{N}\left(U\right)$, we have
$f\circ\varphi\left(\varphi^{-1}\left(U\right)\right)\in\mathcal{O}_{M}\left(\varphi^{-1}\left(U\right)\right)$.
\end{defn}

$\:$
\begin{defn}
\label{AG-def-Nash-manif}(1) A Nash submanifold $M$ of $\mathbb{R}^{n}$
is a semi-algebraic subset of $\mathbb{R}^{n}$ which is a smooth
submanifold. A Nash function on $M$ is a smooth semi-algebraic function.

(2) An affine Nash manifold is an $\mathbb{R}$-space which is isomorphic
to an $\mathbb{R}$-space associated to a closed Nash submanifold
of $\mathbb{R}^{n}$.

(3) A Nash manifold is an $\mathbb{R}$-space $\left(M,\mathcal{N}_{M}\right)$
with a sheaf of Nash functions, which has a finite open cover $\left(M_{i}\right)_{i=1}^{n}$
such that each $\mathbb{R}$-space $\left(M_{i},\mathcal{N}_{M}|_{M_{i}}\right)$
is an affine Nash manifold.
\end{defn}

$\:$

\textbf{Schwartz and tempered functions on Nash manifolds}
\begin{prop}
\label{AG-prop-sub-of-Nash-is-Nash}{[}AG - Proposition 3.3.3{]} (1)
Any open (semi-algebraic) subset $U$ of an affine Nash manifold $M$
with the induced $\mathbb{R}$-space structure is an affine Nash manifold. 

(2) Any open (semi-algebraic) subset $U$ of a Nash manifold $M$
with the induced $\mathbb{R}$-space structure is a Nash manifold.
\end{prop}

\begin{defn}
(1) \textbf{Nash differential operator} on an affine Nash manifold
is an element of the algebra with $1$ generated by multiplication
by Nash functions and derivations along Nash sections of the tangent
bundle (Nash vector fields).

(2) The space of \textbf{Schwartz functions on an affine Nash manifold
$M$} is $\mathcal{S}\left(M\right):=\left\{ \phi\in C^{\infty}\left(M\right)|D\phi\text{ is bounded for any Nash differential operator}\right\} $.

(3) The topology on $\mathcal{S}\left(M\right)$ is defined by the
semi norms $||\phi||{}_{D}:=\sup\limits _{x\in M}|D\phi\left(x\right)|$.

(4) Let $M$ be as in \ref{AG-def-Nash-manif}(3), and $\phi:\bigoplus\limits _{i=1}^{k}\mathcal{S}\left(M_{i}\right)\to C^{\infty}\left(M\right)$
defined by extension by zero and summing. Then $\mathcal{S}\left(M\right):=Im\left(\phi\right)$.
\end{defn}

$\:$
\begin{cor}
\label{AG-prop-Sch-on-Nash-is-Fre}{[}corollary of AG - Corollary
4.1.2{]} Let $M$ be a Nash manifold. Then $\mathcal{S}\left(M\right)$
is a Fréchet space.
\end{cor}

\begin{defn}
{[}AG - Definition 4.2.1 and Theorem 4.6.2{]} A function $t:\mathbb{R}^{n}\to\mathbb{R}$
is called \textit{tempered} if it is a smooth function such that for
any $\alpha\in\left(\mathbb{N}\cup\left\{ 0\right\} \right)^{n}$
there exists a polynomial $p_{\alpha}\in\mathbb{R}\left[x_{1},...,x_{n}\right]$
such that $|\frac{\partial^{|\alpha|}t}{\partial^{\alpha}x}\left(x\right)|<p_{\alpha}\left(x\right)$
for any $x\in\mathbb{R}^{n}$. Let $M$ be an affine Nash manifold,
and let $i:M\hookrightarrow\mathbb{R}^{n}$ be a closed embedding.
A function $t:M\to\mathbb{R}$ is called a \textit{tempered function}
on $M$ if $i_{*}f:=f\circ i^{-1}$ is the restriction to $i(M)$
of a tempered function from $\mathbb{R}^{n}$. Denote the space of
all tempered functions on $M$ by $\mathcal{T}\left(M\right)$. $\mathcal{T}\left(M\right)$
is a well defined space (independent of the embedding chosen).
\end{defn}

\begin{prop}
\label{AG-prop-ts-is-s}{[}corollary of AG - Proposition 4.2.1{]}
Let $M$ be a Nash manifold and $\alpha$ be a tempered function on
$M$. Then $\alpha\mathcal{S}\left(M\right)\subset\mathcal{S}\left(M\right)$.
\end{prop}

\begin{thm}
\label{AG-thm-red-to-cl-is-onto}{[}corollary of AG - Theorem 4.6.1{]}
Let $M$ be a Nash manifold and $Z\hookrightarrow M$ be a closed
Nash submanifold. The restriction $\mathcal{S}\left(M\right)\to\mathcal{S}\left(Z\right)$
is defined, continuous and onto. 
\end{thm}

\begin{prop}
\label{AG-prop-aff-temp-sheaf}{[}AG - Proposition 5.1.3{]} Let $M$
be a Nash manifold. The assignment of the space of tempered functions
on $U$, to any open $U\subset M$, together with the usual restriction
maps, define a sheaf of algebras on $M$.
\end{prop}

\begin{thm}
\label{AG-thm-part-of-unity}{[}AG - Theorem 5.2.1{]} (Partition of
unity for Nash manifold). Let $M$ be a Nash manifold, and let $\left(U_{i}\right)_{i=1}^{n}$
be a finite open cover. Then 

(1) there exist tempered functions $\alpha_{1},...,\alpha_{n}$ on
$M$ such that $supp\left(\alpha_{i}\right)\subset U_{i}$, $\sum\limits _{i=1}^{n}\alpha_{i}=1$. 

(2) Moreover, we can choose $i$ in such a way that for any $\phi\in\mathcal{S}\left(M\right),\:\alpha_{i}\phi\in\mathcal{S}\left(U_{i}\right)$.
\end{thm}

$\:$
\begin{thm}
\label{AG-thm-Char-Nash}{[}AG - Theorem 5.4.1{]} (Characterization
of Schwartz functions on open subset) Let $M$ be a Nash manifold,
$Z$ be a closed (semi-algebraic) subset and $U=M\backslash Z$. Let
$W_{Z}$ be the closed subspace of $\mathcal{S}\left(M\right)$ defined
by $W_{Z}:=\left\{ \phi\in\mathcal{S}\left(M\right)|\phi\text{ vanishes with all its derivatives on }Z\right\} $.
Then restriction and extension by 0 give an isomorphism $\mathcal{S}\left(U\right)\cong W_{Z}$.
\end{thm}

\section{Geometry}
\begin{defn}
\label{Def-NQN}\textit{ ``Naive'' Quasi Nash category - NQN} -
Let $X$ be a locally closed semi-algebraic subset of $\mathbb{R}^{n}$.
A morphism in this category is a map from an NQN set $X$ to an NQN
set $Y\subset\mathbb{R}^{m}$ which is a restriction of a Nash map
on an open semi-algebraic neighborhood $U$ of $X$ to $\mathbb{R}^{m}$
such that $X$ is closed in $U$ and $X$ is mapped into $Y$. i,e,
$\varphi:X\to Y$, $\varphi:=g|_{X}$ where $X\subset U$, and $g:U\to\mathbb{R}^{m}$
is Nash.
\end{defn}

\begin{lem}
\label{NQN-equiv-property}Let $X,\:Y$ be NQN, and let $\varphi:X\to Y$
a continuous map. Then $\varphi$ is an NQN morphism if and only if
for any NQN function $f:Y\to\mathbb{R}$, $f\circ\varphi:X\to\mathbb{R}$
is an NQN function. (NQN function is an NQN map from an NQN set to
$\mathbb{R}$)
\end{lem}

\begin{proof}
Let $\varphi$ be an NQN morphism. Take $f$ as required. As $f$
is a restriction of a Nash function on a neighborhood of $Y$, and
$\varphi$ is a restriction of a Nash map on a neighborhood of $X$,
we get a composition of Nash maps, which is Nash. Thus, $f\circ\varphi$
is an NQN function.

Now let $\varphi$ pullback NQN functions to NQN functions. As $Y\subset\mathbb{R}^{m}$,
$\varphi$ can be presented as $\varphi=\left(\varphi_{1},...,\varphi_{m}\right)$.
For each $i$ take the Nash function $f_{i}=y_{i}$. This results
in Nash functions on neighborhoods $U_{i}$ of $X$ where on each
$U_{i}$: $f_{i}\circ\varphi=y_{i}\circ\varphi=\varphi_{i}$ for any
$i$, what makes all $\varphi$'s coordinates NQN maps. \\
Take the intersection $U=\bigcap\limits _{i=1}^{n}U_{i}$ of all those
neighborhoods to get a neighborhood of $X$ where all of $\varphi$'s
coordinates are NQN simultaneously and therefore $\varphi$ is an
NQN map itself.
\end{proof}
Now let's take a broader category - 
\begin{defn}
\label{Def-Quasy-Nash}(1) Let $X$ be a locally closed semi-algebraic
subset of $\mathbb{R}^{n}$. Define the \textbf{$\mathbb{R}$-space
corresponding to $X$} to be the pair $\left(X,QN_{X}\right)$ where
$QN_{X}$ is defined to be a sheaf as follows: For each open $U\subset X$,
we say $f\in QN_{X}\left(U\right)$ if there exists a collection $\left\{ V_{i}\right\} _{i=1}^{m}$
of open semi-algebraic subsets of $\mathbb{R}^{n}$ and Nash functions
$f_{i}$ from these sets, such that $U_{i}:=V_{i}\cap X$ is an open
cover of $U$, for any $i,\:f|_{U_{i}}=f_{i}|_{U_{i}}$, and for any
$i,j$ $f_{i}|_{U_{i}\cap U_{j}}=f_{j}|_{U_{i}\cap U_{j}}$. This
is a sheafification of the NQN presheaf in definition \ref{Def-NQN}.

(2) An \textbf{affine QN variety $X$} is defined to be an $\mathbb{R}$-space
isomorphic to an $\mathbb{R}$-space obtained by sheafification from
so\textcolor{black}{me closed NQ}N set $Y\subset\mathbb{R}^{n}$.
I.e. $\left(X,O_{X}\right)$ is an affine QN variety if $\left(X,O_{X}\right)\cong\left(Y,O_{Y}\right)$
where $Y$ is \textcolor{black}{a closed }NQN set, and $O_{Y}$ is
the sheafification of the NQN functions on $Y$.
\end{defn}

\begin{rem}
\label{Rem-NQN-Like}Let $X$ be an affine QN variety. Then $X$ is
QN isomorphic to a closed set in $\mathbb{R}^{n}$, which by abuse
of notation is denoted by $X$ as well. Let $V\subset\mathbb{R}^{n}$
be some open semi-algebraic subset such that $U=V\cap X$ is an open
subset of $X$. Then $V$ is Nash diffeomorphic to some closed affine
Nash submanifold $V'$ in $\mathbb{R}^{N}$, by {[}BCR - Theorems
8.4\textcolor{black}{.6, 2.4.5{]}. Nash diffeomorphism is a QN isomorphism
(i.e. an isomorphism of ringed spaces) so $V$ is QN isomorphic to
$V'$. Denote the image of $U\subset V$ by $U'$, and note that it
is closed in $V'$. This makes $U'$ a closed subset of $\mathbb{R}^{N}$,
and so an NQN object. A Nash function on $V$ can be pulled-back to
a Nash function on $V'$, and a Nash function on $V'$ can be extended
to a Nash function on an open neighborhood of $V'$ (e.g. by {[}AG
- Thm 3.6.2 - Nash Tubular Neighborhood{]}). Thus, we can say that
a restriction of a Nash function from $V$ to $U$ is QN isomorphic
to an NQN function on the NQN set $U'$, which is QN isomorphic to
$U$. This means a locally closed semi-algebraic set of $\mathbb{R}^{n}$
has a natural stracture of an affine QN variety, and its $\mathbb{R}$-space
functions are restrictions of Nash functions from an open semi-algebraic
sets in which the affine QN-variety is closed. The same goes for locally
closed subsets of a QN variety.}
\end{rem}

\begin{lem}
Affine Nash manifolds form a full subcategory of affine QN varieties.
Affine algebraic varieties form a subcategory, not a full subcategory.
\end{lem}

\begin{proof}
An affine Nash manifold is an $\mathbb{R}$-space isomorphic to an
$\mathbb{R}$-space associated to a smooth closed semi-algebraic subset
of $\mathbb{R}^{n}$. The sheaf on a Nash manifold is made of Nash
functions - smooth semi-algebraic functions. In the QN category the
sheaf is made of locally restrictions of Nash functions. Due to the
Nash tubular neighborhood, the sheaves are the same on Nash manifolds.
Let $\varphi:X\to Y$ be a QN morphism between the Nash manifolds
$X,\:Y$. Take a Nash function $f$ on $Y$. So there exists a cover
$\bigcup\limits _{i=1}^{n}X_{i}=X$ s.t. $g_{i}:=f\circ\varphi|_{X_{i}}$
is a Nash function. As Nash functions on the Nash manifold $X$ form
a sheaf, $\exists g\in\mathcal{N}\left(X\right)$ s.t. $g|_{X_{i}}=g_{i}$.
So for any QN map $\varphi$, and any $f\in\mathcal{N}\left(Y\right)$,
we get $f\circ\varphi\in\mathcal{N}\left(X\right)$.\\
Affine algebraic variety is isomorphic to an algebraic subset of $\mathbb{R}^{n}$.
But as morphisms in this category are rational maps, and not any Nash
map is such a map (e.g. $\sqrt{1+x^{2}}$), this is not a full subcategory.
\end{proof}
\begin{lem}
\label{lem-affine-QN-equiv}Let $X,\:Y$ be affine QN varieties and
$\tilde{X}\subset\mathbb{R}^{n},\:\tilde{Y}\subset\mathbb{R}^{m}$
the corresponding closed QN sets. Let $\varphi:X\to Y$ be a continuous
map and $\tilde{\varphi}:\tilde{X}\to\tilde{Y}$ the corresponding
map. Then $\varphi$ is a QN morphism if and only if there exists
an open cover $\bigcup\limits _{i=1}^{N}\tilde{X}_{i}=\tilde{X}$
such that $\tilde{\varphi}|_{\tilde{X}_{i}}$ is an NQN map for any
$i$.
\end{lem}

\begin{proof}
\textcolor{black}{Assume $\varphi$ is an affine QN morphism. For
any $f\in QN_{\tilde{Y}},\:\tilde{\varphi}^{*}f\in QN_{\tilde{X}}$.
}As $\tilde{Y}\subset\mathbb{R}^{m}$, $\tilde{\varphi}$ can be presented
as $\tilde{\varphi}=\left(\tilde{\varphi}_{1},...,\tilde{\varphi}_{m}\right)$.
For each $j\in\left\{ 1,...,m\right\} $ take the Nash function $f_{j}=\tilde{y}_{j}$.
When pulling it back, we get an open cover $\tilde{X}=\bigcup\limits _{k=1}^{N_{j}}\tilde{X}_{k}$
such that for any open subset $\tilde{X}_{k}$, the function $f_{j}\circ\tilde{\varphi}=\tilde{y}_{j}\circ\tilde{\varphi}=\tilde{\varphi}_{j}$
is a restriction of a Nash function, i.e. $\tilde{\varphi}_{j}|_{\tilde{X}_{k}}$
is an NQN map for any $k$. Now take the refinement of those $j$
covers of $\tilde{X}$, and denote it by $\tilde{X}=\bigcup\limits _{i=1}^{N}\tilde{X}_{i}$.
As a restriction of a Nash function is a Nash function, we get that
for any $i$, $\left(\tilde{\varphi}_{1},...,\tilde{\varphi}_{m}\right)|_{\tilde{X}_{i}}$
is an NQN map, what makes $\tilde{\varphi}|_{\tilde{X}_{i}}$ an NQN
map.

Now the other direction:

Assume $\tilde{\varphi}|_{\tilde{X}_{i}}$ is an NQN map for any $i$
and let $f\in QN_{Y}\left(U\right)$ for an open $U\subset Y$. Those
$f$ and $U$ correspond to some $\tilde{f}$ and $\tilde{U}$ respectively.
This means there is an open cover $\tilde{U}=\bigcup\limits _{j=1}^{M}\tilde{U}_{j}$
and for each $j$, $\tilde{f}|_{\tilde{U}_{j}}$ is a restriction
of a Nash function. The map $\tilde{\varphi}$ is continuous, so $\tilde{\varphi}|_{\tilde{X}_{i}\cap\tilde{\varphi}^{-1}\left(\tilde{U}_{j}\right)}:\tilde{X}_{i}\cap\tilde{\varphi}^{-1}\left(\tilde{U}_{j}\right)\to\tilde{U}$
is an NQN map, i.e. it is a restriction of a Nash map from an open
neighborhod of $\tilde{X}_{i}\cap\tilde{\varphi}^{-1}\left(\tilde{U}_{j}\right)$.
Therefore $\tilde{f}\circ\tilde{\varphi}|_{\tilde{X}_{i}\cap\tilde{\varphi}^{-1}\left(\tilde{U}_{j}\right)}$
is a restriction of a Nash function for any $j$ and any $i$. Thus,
$f\circ\varphi\in QN_{X}\left(\varphi^{-1}\left(U\right)\right)$.
\end{proof}
\begin{defn}
\label{def-gen-QN-variety}A \textbf{QN variety} is an $\mathbb{R}$-space
$\left(X,QN_{X}\right)$ where $X$ is a restricted topological space,
which has a finite cover $\bigcup\limits _{i=1}^{m}X_{i}=X$ of open
sets such that the $\mathbb{R}$-spaces $\left(X_{i},QN_{X}|_{X_{i}}\right)$
are affine QN varieties.
\end{defn}

\begin{lem}
\label{lem-gen-QN-equiv}Let $X,\:Y$ be general QN varieties. Let
$\varphi:X\to Y$ be a continuous map. Then $\varphi$ is a QN morphism
if and only if there exists an open affine cover $Y=\bigcup\limits _{i=1}^{M}Y_{i}$,
and an open \textcolor{black}{affine} cover $\bigcup\limits _{j=1}^{N_{i}}X_{ij}=X_{i}:=\varphi^{-1}\left(Y_{i}\right)$,
where each $X_{ij}$ corresponds to a locally closed semi-algebraic
set $\tilde{X}_{ij}$ such that the induced map $\tilde{\varphi}|_{\tilde{X}_{ij}}$
is an NQN map to the closed semi-algebraic set $\tilde{Y}_{i}$ isomorphic
to $Y_{i}$.
\end{lem}

\begin{proof}
\textcolor{black}{Assume $\varphi$ is a QN morphism. Take an open
affine cover }$Y=\bigcup\limits _{i=1}^{M}Y_{i}$, and for each $i$
take the QN variety $X_{i}:=\varphi^{-1}\left(Y_{i}\right)$. Cover
it by some open affine subvarieties $\bigcup\limits _{j=1}^{N_{i}}X^{ij}=X_{i}$.
Now we have a restricted morphism $\varphi|_{X^{ij}}:X^{ij}\to Y_{i}$
of affine QN varieties and we can use Lemma \ref{lem-affine-QN-equiv}.
Refining the covers will yield the proof for the first direction.

To prove the other direction we start with a function $f\in QN_{Y}\left(U\right)$
for an open $U\subset Y$. By remark \ref{Rem-NQN-Like} each $X_{ij}$
corresponds to an affine QN variety, so by Lemma \ref{lem-affine-QN-equiv}
we have QN morphisms $\text{\ensuremath{\varphi}}|_{X_{ij}}:X_{ij}\to Y_{i}$.
Thus, $f|_{Y_{i}\cap U}\in QN_{Y}\left(Y_{i}\cap U\right)$ is pulled
back to $f\circ\varphi|_{X_{ij}\cap\varphi^{-1}\left(U\right)}\in QN_{X}\left(X_{ij}\cap\varphi^{-1}\left(U\right)\right)$
for any $i$, and $j$. As QN functions form a sheaf, we get that
$f\circ\varphi|_{\varphi^{-1}\left(U\right)}\in QN_{X}\left(\varphi^{-1}\left(U\right)\right)$. 
\end{proof}

\section{Schwartz Functions, Tempered Functions and Tempered Distributions}

\subsection{Naive Quasi-Nash}
\begin{claim}
\label{claim-Sch-from-open}Let $X$ be a NQN set, and $U,\:V\subset\mathbb{R}^{n}$
be open sets containing $X$ as a closed subset. Then $\mathcal{S}\left(U\right)/I_{Sch}^{U}\left(X\right)\cong\mathcal{S}\left(V\right)/I_{Sch}^{V}\left(X\right)$.
\end{claim}

\begin{proof}
We start by showing $\mathcal{S}\left(U\right)/I_{Sch}^{U}\left(X\right)\cong\mathcal{S}\left(U\cap V\right)/I_{Sch}^{U\cap V}\left(X\right)$.
By Theorem \ref{AG-thm-Char-Nash}, $\mathcal{S}\left(U\cap V\right)$
is isomorphic to a closed subspace of $\mathcal{S}\left(U\right)$.
Thus, it is enough to check that ${\color{black}{\color{black}\mathcal{S}\left(X\right)}:=\mathcal{S}\left(U\right)/I_{Sch}^{U}\left(X\right)}$
and $\mathcal{S}\left(U\cap V\right)/I_{Sch}^{U\cap V}\left(X\right)$
are equal as sets, i.e. that a Schwartz function on $X$ is a restriction
of a Schwartz function on $U$ if and only if it is a restriction
of a Schwartz function on $U\cap V$. Let $f\in\mathcal{S}\left(U\cap V\right)/I_{Sch}^{U\cap V}\left(X\right)$.
There exists $F\in\mathcal{S}\left(U\cap V\right)$ such that $F|_{X}=f$.
By Theorem \ref{AG-thm-Char-Nash}, extending $F$ by zero to a function
on $U$ (denote it by $\tilde{F}$) is a function in $\mathcal{S}\left(U\right)$.
Then $f=\tilde{F}|_{X}$ and so $f\in\mathcal{S}\left(U\right)/I_{Sch}^{U}\left(X\right)$.
For the other direction, let $f\in\mathcal{S}\left(U\right)/I_{Sch}^{U}\left(X\right)$.
There exists $F\in\mathcal{S}\left(U\right)$ such that $F|_{X}=f$.
Denote $U'=U\backslash X$. $\left\{ U',U\cap V\right\} $ form an
open cover of $U$ and so, by Theorem \ref{AG-thm-part-of-unity},
there exist tempered functions $\alpha_{1},\:\alpha_{2}$ such that
$supp\left(\alpha_{1}\right)\subset U\cap V,\:supp\left(\alpha_{2}\right)\subset U'$
and $\alpha_{1}+\alpha_{2}=1$ as a real valued function on $U$.
Moreover, $\alpha_{1}$ and $\alpha_{2}$ can be chosen such that
$\left(\alpha_{1}\cdot F\right)|_{U\cap V}\in\mathcal{S}\left(U\cap V\right)$.
As $\alpha_{1}|_{X}=1$, it follows that $\left(\left(\alpha_{1}\cdot F\right)|_{U\cap V}\right)|_{X}=\left(\alpha_{1}\cdot F\right)|_{X}=F|_{X}=f$,
and so $f\in\mathcal{S}\left(U\cap V\right)/I_{Sch}^{U\cap V}\left(X\right)$.

As we can use the same proof exactly to show 
\[
\mathcal{S}\left(V\right)/I_{Sch}^{V}\left(X\right)\cong\mathcal{S}\left(U\cap V\right)/I_{Sch}^{U\cap V}\left(X\right),
\]
we end up having the result.
\end{proof}
\begin{claim}
\label{claim-temp-from-open}Let $X$ be a NQN set, and $U,\:V\subset\mathbb{R}^{n}$
be open sets containing $X$ as a closed subset. Then for any $\mathcal{T}\left(U\right)|_{X}=\mathcal{T}\left(V\right)|_{X}$.
\begin{proof}
As $U,\:V$ are open Nash submnifolds (Preposition \ref{AG-prop-sub-of-Nash-is-Nash}),
as tempered functions form a sheaf on $U\cup V$ (by Proposition \ref{AG-prop-aff-temp-sheaf})
and as we may use tempered partition of unity on $U\cup V$ (Theorem
\ref{AG-thm-part-of-unity}), the claim easily follows.
\end{proof}
\end{claim}

\begin{defn}
\textit{Let $X$ be an NQN set. A Schwartz function on $X$} is a
restriction of a Schwartz function from an open semi-algebraic subset
of $\mathbb{R}^{n}$ in which $X$ is closed, to $X$. Equivalently,
we can define the space of Schwartz functions on $X$ as $\mathcal{S}\left(X\right):=\mathcal{S}\left(U\right)/I_{Sch}^{U}\left(X\right)$.

\textit{A tempered function on $X$} is a restriction of a tempered
function from an open semi-algebraic subset of $\mathbb{R}^{n}$ in
which $X$ is closed, to $X$. The space of tempered functions on
$X$ is denoted by $\mathcal{T}\left(X\right)$.
\end{defn}

\begin{lem}
\label{lem-Sch-NQN-is-Frechet}Let $X$ be an NQN set. Then $\mathcal{S}\left(X\right)$
is a Fréchet space.
\end{lem}

\begin{proof}
$\mathcal{S}\left(U\right)$ is a Fréchet space (see {[}Proposition
\ref{AG-prop-sub-of-Nash-is-Nash}, Proposition \ref{AG-prop-Sch-on-Nash-is-Fre}{]})
and as $I_{Sch}^{U}$ is a closed subset of $\mathcal{S}\left(U\right)$,
we get that their quotient is a Fréchet space as well (by Proposition
\ref{prop-Closed-Frechet-subspace} and {[}T - Proposition 7.9{]}).
\end{proof}
\begin{lem}
\label{lem-NQN-Frechet-iso}Let $\varphi:X_{1}\to X_{2}$ be an NQN
isomorphism, i.e. a bijective morphism whose inverse is also an NQN
morphism. Then $\varphi^{*}|_{\mathcal{S}\left(X_{2}\right)}:\mathcal{S}\left(X_{2}\right)\to\mathcal{S}\left(X_{1}\right)$
is an isomorphism of Fréchet spaces.
\end{lem}

\begin{proof}
By definition we have an open semialgebraic neighborhood $U_{1}$
of $X_{1}$ and a Nash map $g_{1}:U_{1}\to\mathbb{R}^{n_{2}}$ such
that $g_{1}|_{X}=\varphi$. Note that $U_{1}$ is an affine Nash manifold.

Similarly to the construction of $U_{1}$ and $g_{1}$ above, we may
construct an open $U_{2}\subset\mathbb{R}^{n_{2}}$ and a map $g_{2}:U_{2}\to\mathbb{R}^{n_{1}}$
such that $g_{2}|_{X_{2}}=\varphi{}^{-1}$. Note that $g_{2}\neq g_{1}^{-1}$:
in general $g_{1}$ is not a bijection and $U_{1}\ncong U_{2}$.

Consider the following diagram, where $\alpha$ is defined by $\alpha(x,y):=(x,y+g_{1}(x))$.

\includegraphics{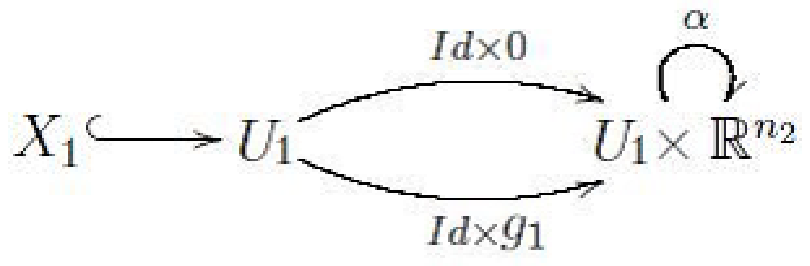}

Clearly $U_{1}\times\{0\}$ is an affine Nash manifold isomorphic
to $U_{1}$. Denote $\hat{U}_{1}:=\alpha(U_{1}\times\{0\})$, then
$\alpha$ restricted to $U_{1}\times\{0\}$ is an isomorphism of the
affine Nash manifolds $U_{1}\times\{0\}$ and $\hat{U}_{1}$ \textendash{}
the inverse map is given by $\alpha^{-1}(x,y):=(x,y-g_{1}(x))$. Thus
we have: 
\[
\mathcal{S}(X_{1})\cong\mathcal{S}(U_{1})/I_{Sch}^{U}(X_{1})\cong\mathcal{S}(\hat{U}_{1})/I_{Sch}^{\hat{U}_{1}}(\alpha(X_{1}\times\{0\}))=\mathcal{S}(\hat{U}_{1})/I_{Sch}^{\hat{U}_{1}}((Id\times\varphi)(X_{1})),
\]

where the first equivalence is by definition, the second is due the
fact that $U_{1}\cong U_{1}\times\{0\}\cong\hat{U}_{1}$ and $\mathcal{S}(U_{1})\cong\mathcal{S}(U_{1}\times\{0\})\cong\mathcal{S}(\hat{U}_{1})$,
and the third follows from the fact that $g_{1}|_{X_{1}}=\varphi$.
As always $I_{Sch}^{U_{1}}(X_{1})$ is the ideal in $\mathcal{S}(U_{1})$
of Schwartz functions identically vanishing on $X_{1}$.

As $\tilde{U_{1}}$ is closed in $U_{1}\times\mathbb{R}^{n_{2}}$
(as it is defined by Nash map on $U_{1}\times\mathbb{R}^{n_{2}}$),
then by Theorem \ref{AG-thm-red-to-cl-is-onto} we get that 
\[
\mathcal{S}(\hat{U}_{1})/I_{Sch}^{\hat{U}_{1}}((Id\times\varphi)(X_{1}))\cong\mathcal{S}(U_{1}\times\mathbb{R}^{n_{2}})/I_{Sch}^{U_{1}\times\mathbb{R}^{n_{2}}}((Id\times\varphi)(X_{1})).
\]

Applying claim \ref{claim-Sch-from-open} for the open subset $U_{1}\times U_{2}\subset U_{1}\times\mathbb{R}^{n_{2}}$
we get that 
\[
\mathcal{S}(U_{1}\times\mathbb{R}^{n_{2}})/I_{Sch}^{U_{1}\times\mathbb{R}^{n_{2}}}((Id\times\varphi)(X_{1}))\cong\mathcal{S}(U_{1}\times U_{2})/I_{Sch}^{U_{1}\times U_{2}}((Id\times\varphi)(X_{1})),
\]

and thus we obtain
\[
\mathcal{S}(X_{1})\cong\mathcal{S}(U_{1}\times U_{2})/I_{Sch}^{U_{1}\times U_{2}}((Id\times\varphi)(X_{1})).
\]

Repeating the above construction using the following diagram:

\includegraphics{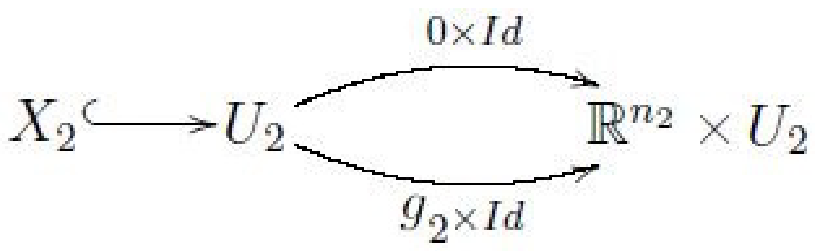}

yields:
\[
\mathcal{S}(X_{2})\cong\mathcal{S}(U_{1}\times U_{2})/I_{Sch}^{U_{1}\times U_{2}}((\varphi^{-1}\times Id)(X_{2})).
\]

Clearly $(Id\times\varphi)(X_{1})=(\varphi^{-1}\times Id)(X_{2})$,
and so $\mathcal{S}(X_{1})\cong\mathcal{S}(X_{2})$. Note that the
isomorphism constructed is in fact the pull back by $\varphi$ from
$\mathcal{S}(X_{2})$ onto $\mathcal{S}(X_{1})$. This proves the
lemma.
\end{proof}
\begin{rem}
A similar claim about tempered functions can be proved similarly,
by replacing Theorem \ref{AG-thm-red-to-cl-is-onto} with a similar
claim about tempered functions ({[}AG - 4.6.2{]}).
\begin{prop}
\label{prop-sXt}Let $X$ be an NQN set, $s\in\mathcal{S}\left(X\right)$,
and $t\in\mathcal{T}\left(X\right)$. Then $t\cdot s\in\mathcal{S}\left(X\right)$.
\end{prop}

\begin{proof}
Take some $U\subset\mathbb{R}^{n}$ such that $X\subset U$ as a closed
subset. Then there exist some Schwartz function $S\in\mathcal{S}\left(U\right)$
and a tempered function $T\in\mathcal{T}\left(U\right)$ such that
$s=S|_{X}$ and $t=T|_{X}$. As $T\cdot S\in\mathcal{S}\left(U\right)$
(by Proposition \ref{AG-prop-ts-is-s}), we get that $\left(T\cdot S\right)|_{X}\in\mathcal{S}\left(X\right)$.
\end{proof}
\end{rem}

$\:$

\subsection{Affine QN}
\begin{lem}
Let $X,\:Y$ be affine QN varieties, and $\varphi:X\to Y$ a QN isomorphism.
Let $\tilde{X},\:\tilde{Y}$ be closed NQN sets correponding to $X,\:Y$
respectively, and $\tilde{\varphi}$ the corresponding map. Then a
Schwartz function on $\tilde{Y}$ is pulled-back by $\tilde{\varphi}$
to a Schwartz function on $\tilde{X}$.
\end{lem}

\begin{proof}
Let $f\in\mathcal{S}\left(\tilde{Y}\right)$. By Lemma \ref{lem-affine-QN-equiv}
there exist open covers $\bigcup\limits _{i=1}^{N}\tilde{X}_{i}=\tilde{X}$,
$\bigcup\limits _{i=1}^{N}\tilde{Y}_{i}=\tilde{Y}$ such that $\tilde{\varphi}|_{\tilde{X}_{i}}:\tilde{X}_{i}\xrightarrow{\sim}\tilde{Y}_{i}$
are NQN isomorphisms. As Schwartz functions form a cosheaf, every
Schwartz function on $\tilde{Y}$ is a sum of extensions of Schwartz
functions on open subsets of $\tilde{Y}$. Let those subsets be $\tilde{Y}_{i}$.
According to Lemma \ref{lem-NQN-Frechet-iso}, a Schwartz function
$s_{i}\in\mathcal{S}\left(\tilde{Y}_{i}\right)$ is pulled back to
$\varphi^{*}s_{i}\in\mathcal{S}\left(\tilde{X}_{i}\right)$, and thus,
we get $\varphi^{*}s=\sum\limits _{i=1}^{n}Ext_{\tilde{X}_{i}}^{\tilde{X}}\left(\varphi^{*}s_{i}\right)\in\mathcal{S}\left(\tilde{X}\right)$.
Together with the NQN isomorphism on each such subset, we get the
result.
\end{proof}
This lemma enables us to use the following definition:
\begin{defn}
Let $X$ be an affine QN variety. A function $s$ on $X$ is called
a \textit{Schwartz function on $X$} if it is a pullback of a Schwartz
function from the closed NQN set corresponds to $X$. I.e. if $\varphi:X\xrightarrow{\sim}\tilde{X}\subset\mathbb{R}^{n}$
where $\tilde{X}$ is the closed semi-algebraic set QN isomorphic
to $X$, and $\tilde{s}\in\mathcal{S}\left(\tilde{X}\right)$, then
$\varphi^{*}\tilde{s}\in\mathcal{S}\left(X\right)$.
\end{defn}

\begin{rem}
\textcolor{black}{\label{rem-NQN-Like-Sch-Tempered}Let $X$ be a
locally-closed subset in $\mathbb{R}^{n}$. Let $U\subset\mathbb{R}^{n}$
be an open neighborhood of $X$ such that $X$ is closed in $U$.
$U$ is an affine Nash manifold, so there is a Nash diffeomotphism
$\phi$ between $U$ and a closed Nash submanifold of $\mathbb{R}^{N}$.
Denote $\hat{X}:=\phi\left(X\right)$ and $\hat{U}:=\phi\left(U\right)$.
Nash diffeomorphism is a closed map, so $\hat{X}\subset\hat{U}$ is
a closed subset. Thus, we may define the Schwartz space $\mathcal{S}\left(\hat{X}\right):=\mathcal{S}\left(\mathbb{R}^{N}\right)|_{\hat{X}}=\left(\mathcal{S}\left(\mathbb{R}^{N}\right)|_{\hat{U}}\right)|_{\hat{X}}=\mathcal{S}\left(\hat{U}\right)|_{\hat{X}}$
. As Nash diffeomorphisms pull back Schwartz functions, the space
$\mathcal{S}\left(\hat{U}\right)$ is pulled back to $\mathcal{S}\left(U\right)$,
and $\mathcal{S}\left(\hat{X}\right)$ is pulled back to $\mathcal{S}\left(X\right)$
and vice versa. This enables us to understand a Schwartz function
on a loaclly-closed subset $X$ of $\mathbb{R}^{n}$ as a restriction
of a Schwartz function from an open set $U$ in which $X$ is closed.
I.e.: $\mathcal{S}\left(X\right)=\mathcal{S}\left(U\right)|_{X}$.}

A similar claim for tempered functions can be proven the same way.
\end{rem}

\begin{thm}
\label{thm-Res To Closed}Let $M$ be an affine QN variety, and let
$X\subset M$ be a closed subset. Then the restriction of a Schwartz
function from $M$ to $X$ defines an isomorphism $\mathcal{S}\left(X\right)=\mathcal{S}\left(M\right)/I_{Sch}^{M}\left(X\right)$
(with the quotient topology), where $I_{Sch}^{M}\left(X\right)$ is
the ideal in $\mathcal{S}\left(M\right)$ of functions identically
vanishing on $X$.
\end{thm}

\begin{proof}
Take the closed corresponding sets of $X$ and $M$, $\tilde{X}$
and $\tilde{M}$ correspondingly. As $\tilde{X}\subset\tilde{M}\subset U$
where $\tilde{M}$ is closed in some open $U\subset\mathbb{R}^{n}$,
we get that $\mathcal{S}\left(M\right)/I_{Sch}^{M}\left(X\right)\cong\left(\mathcal{S}\left(U\right)/I_{Sch}^{U}\left(M\right)\right)/I_{Sch}^{M}\left(X\right)\cong\mathcal{S}\left(U\right)/I_{Sch}^{U}\left(X\right)\cong\mathcal{S}\left(X\right)$.
\end{proof}
\begin{defn}
Let $X$ be an affine QN variety corresponding to some closed subset
$\tilde{X}\subset\mathbb{R}^{n}$. A function $f:\tilde{X}\to\mathbb{R}$
is called \textbf{flat} at $\tilde{p}\in\tilde{X}$ if there exists
an open semi-algebraic neighborhood $U\subset\mathbb{R}^{n}$ of $\tilde{p}$
and a function $F\in C^{\infty}\left(U\right)$ such that $f|_{U\cap\tilde{X}}=F|_{U\cap\tilde{X}}$
and $F$'s Taylor series is identically zero at $\tilde{p}$. If $f$
is flat at all $\tilde{p}\in\tilde{Z}$ for some $\tilde{Z}\subset\tilde{X}$,
$f$ is called \textbf{flat at $\tilde{Z}$}.
\end{defn}

\begin{claim}
Let $X$ and $Y$ be affine QN varieties, and $\varphi:X\to Y$ a
QN isomorphism. Let $\tilde{X}\subset\mathbb{R}^{n}$ and $\tilde{Y}\subset\mathbb{R}^{m}$
be the corresponding closed subsets, and $\tilde{\varphi}:\tilde{X}\to\tilde{Y}$
the corresponding map. Let $f$ be a function on $\tilde{Y}$ that
is flat at some $p\in\tilde{Y}$. Then $\tilde{\varphi}^{*}f$ is
flat at $\tilde{\varphi}^{-1}\left(p\right)$.
\end{claim}

\begin{proof}
By definition there is some open neighborhood $U\subset\mathbb{R}^{m}$
of $p$ and a smooth function $F\in C^{\infty}\left(U\right)$ which
is flat on $p$ and such that $f=F|_{\tilde{Y}\cap U}$. As $\tilde{\varphi}$
is an isomorphism, and by Lemma \ref{lem-affine-QN-equiv}, $\tilde{\varphi}^{-1}\left(p\right)$
has an open neighborhood $W\subset\mathbb{R}^{n}$ such that $\tilde{\varphi}|_{\tilde{X}\cap W}$
is a restriction of a Nash map $G:W\to\mathbb{R}^{m}$. Take an open
subset $W'\subset W$ such that $p\in W'$ and $G\left(W'\right)\subset U$.
By sheaf properties of Nash maps, $G|_{W'}$ is Nash as well. Thus,
$F\circ G|_{W'}$ is a smooth function on $W'$, and by {[}Lemma \ref{lem-Faa-di}{]}
it is flat on $\tilde{\varphi}^{-1}\left(p\right)$. We conclude that
$F\circ G|_{\tilde{X}\cap W'}=F\circ\tilde{\varphi}|_{\tilde{X}\cap W'}=f\circ\tilde{\varphi}|_{\tilde{X}\cap W'}$,
which means $\tilde{\varphi}^{*}f$ is flat at $\tilde{\varphi}^{-1}\left(p\right)$.
\end{proof}
\begin{defn}
Let $X$ be an affine QN variety. A function $f:X\to\mathbb{R}$ is
called flat at $p\in X$ if the correponding map $\tilde{f}$ is flat
at the corresponding point $\tilde{p}$. If $f$ is flat at all $p\in Z$
for some $Z\subset X$, $f$ is called \textbf{flat at $Z$}.
\end{defn}

\begin{lem}
\label{lem-ext-by-0-affine}(Extension by zero - the affine case)
Let $X$ be an affine QN variety, and let $V\subset X$ be an open
subset. Then any $f\in\mathcal{S}\left(V\right)$ can be extended
to a Schwartz function on $X$ which is flat on $X\backslash V$.
\end{lem}

\begin{proof}
First, let us take the closed set in $\mathbb{R}^{n}$ corresponding
to $X$, and denote it, by abuse of notation, by $X$. By remark \ref{rem-NQN-Like-Sch-Tempered},
$f$ is a restriction of a Schwartz function $\hat{f}$ on some open
neighborhood $\hat{V}\subset\mathbb{R}^{n}$. By Theorem \ref{AG-thm-Char-Nash}
$\hat{f}$ may be extended by zero to a Schwartz function on $\mathbb{R}^{n}$
which is flat on $\mathbb{R}^{n}\backslash\hat{V}$. Restricting this
function to $X$ yields the desired result. 
\end{proof}
In fact, for any open $V\subset X$, any restriction to $V$ of a
Schwartz function on $X$ which is flat on $X\backslash V$, is a
Schwartz function on $V$. In order to prove that, we need the following
lemmas.

The following lemmas and proposition are required for the proof of
Theorem \ref{thm-affine-char}:
\begin{lem}
\label{lem-subanalytic-ext}Let $X$ be a QN variety corresponding
to some compact set $\tilde{X}\subset\mathbb{R}^{n}$ closed in some
$U\subset\mathbb{R}^{n}$. Let $Z\subset\tilde{X}$ be a closed subset.
Define $V:=\tilde{X}\backslash Z$, 
\[
W_{Z}:=\left\{ \phi:\tilde{X}\to\mathbb{R}|\exists\hat{\phi}\in C^{\infty}\left(U\right)\text{ such that }\hat{\phi}|_{\tilde{X}}=\phi\text{ and }\phi\text{ is flat at }Z\right\} ,
\]
 and 
\[
\left(W_{Z}^{U}\right)^{comp}:=\left\{ \phi\in C^{\infty}\left(U\right)|\phi\text{ is compactly supported and is flat at }Z\right\} .
\]
Then, for any $f\in W_{Z}$ there exists $\hat{f}\in\left(W_{Z}^{U}\right)^{comp}$
such that $\hat{f}|_{X}=f$.
\end{lem}

\begin{proof}
The proof of Lemma \ref{lem-subanalytic-ext} is exactly the same
as the proof of {[}ES - Lemma 3.13{]}, which deals with algebraic
varieties. In the case of $Z=\left\{ p\right\} $, this extension
is trivial. In the case $Z$ consists of more than one point, we need
to use results on Whitney's extension theorem. We used {[}BM1{]},
{[}BM2{]}, {[}BMP1{]}, dealing with subanalytic geometry, to prove
this extension exists in the algebraic case in {[}ES - Lemma 3.13
and Appendix A{]}. As subanalytic geometry covers semi-algebraic sets
as well, the proof is valid in our case with these minor changes:\\
Semi-algebraic sets replace algebraic sets, affine QN varieties replace
affine algebraic varieties, and open\textbackslash{}closed semi-algebraic
sets replace Zariski open\textbackslash{}closed sets.\\
Restricting a Schwartz function to an affine QN variety from an open
neighborhood is given in Remark \ref{rem-NQN-Like-Sch-Tempered},
which replaces {[}ES - Theorem 3.7{]}.\\
The Uniformization theorem - {[}BM1, 5.1{]} should be replaced by
the more general Theorem {[}BM1, Theorem 0.1{]}.
\end{proof}
\begin{lem}
\label{lem-Sch-are-smth-on-comp}Let $X$ be an affine QN variety
corresponding to some compact set $\tilde{X}\subset\mathbb{R}^{n}$.
Let $U\subset\mathbb{R}^{n}$ such that $\tilde{X}$ is closed in
$U$. Then 
\[
\mathcal{S}\left(\tilde{X}\right)=\left\{ f:\tilde{X}\to\mathbb{R}|\exists\hat{f}\in C^{\infty}\left(U\right)\text{ such that }\hat{f}|_{\tilde{X}}=f\right\} 
\]
\end{lem}

\begin{proof}
The inclusion $\subset$ is trivial as $\mathcal{S}\left(U\right)\subset C^{\infty}\left(U\right)$.
For the other direction consider some $g:\tilde{X}\to\mathbb{R}$
and assume it extends to some $\hat{g}\in C^{\infty}\left(U\right)$
such that $\hat{g}|_{\tilde{X}}=g$. Let $\rho\in C^{\infty}\left(U\right)$
be a compactly supported function such that $\rho|_{\tilde{X}}=1$.
Then $\rho\cdot\hat{g}$ is a smooth compactly supported function
on $U$, so $\rho\cdot\hat{g}\in\mathcal{S}\left(U\right)$. Moreover,
$\left(\rho\cdot\hat{g}\right)|_{\tilde{X}}=\hat{g}|_{\tilde{X}}=g$.
Thus $g\in\mathcal{S}\left(\tilde{X}\right)$.
\end{proof}
\begin{prop}
\label{prop-res-Wz-to-U-is-Sch}Let $X$ be an affine QN variety corresponding
to some closed $\tilde{X}$, and let $Z\subset\tilde{X}$ be some
closed subset. Define $V:=\tilde{X}\backslash Z$ and 
\[
W_{Z}:=\left\{ \phi\in\mathcal{S}\left(\tilde{X}\right)|\phi\text{ is flat on }Z\right\} .
\]
Then restriction from $\tilde{X}$ to $V$ of a function in $W_{Z}$
is a Schwartz function on $V$, i.e. $Res_{V}^{\tilde{X}}\left(W_{Z}\right)\subset\mathcal{S}\left(V\right)$.
\end{prop}

\begin{proof}
The proof is divided into two parts. First we show the case where
$X$ corresponds to a compact subset $\tilde{X}\subset\mathbb{R}^{n}$.
Then, we deduce the general case.

Assume $\tilde{X}$ is compact. Define $V^{\mathbb{R}^{n}}:=\mathbb{R}^{n}\backslash Z$
and 
\[
W_{Z}^{\mathbb{R}^{n}}:=\left\{ \phi\in\mathcal{S}\left(\mathbb{R}^{n}\right)|\phi\text{ is flat on }Z\right\} .
\]
As $Z$ is closed in $\mathbb{R}^{n}$, $V^{\mathbb{R}^{n}}$ is open.
As $V=V^{\mathbb{R}^{n}}\cap X$, $V$ is closed in $V^{\mathbb{R}^{n}}$.
The claim follows from the existence of these three maps:
\[
\xymatrix{ & W_{Z}^{\mathbb{R}^{n}}\ar@{->>}[dl]_{Res_{X}^{\mathbb{R}^{n}}}^{\left(1\right)}\ar[dr]_{\left(2\right)}^{Res_{U^{\mathbb{R}^{n}}}^{\mathbb{R}^{n}}}\\
W_{Z}\ar@{-->}[dr]_{Res_{U}^{X}} &  & \mathcal{S}\left(U^{\mathbb{R}^{n}}\right)\ar[dl]_{\left(3\right)}^{Res_{U}^{U^{\mathbb{R}^{n}}}}\\
 & \mathcal{S}\left(U\right)
}
\]

The existence of map (1) is clear. It is onto due to Lemma \ref{lem-subanalytic-ext}
and Lemma \ref{lem-Sch-are-smth-on-comp}. Let $g\in W_{Z}^{\mathbb{R}^{n}}$.
Then we get map (2) by Theorem \ref{AG-thm-Char-Nash}, $g|_{V^{\mathbb{R}^{n}}}\in\mathcal{S}\left(V^{\mathbb{R}^{n}}\right)$.
Map (3) is obtained as for any $h\in\mathcal{S}\left(V^{\mathbb{R}^{n}}\right)$
we get $h|_{V}\in\mathcal{S}\left(V\right)$ by Remark \ref{rem-NQN-Like-Sch-Tempered}. 

Now assume $\tilde{X}$ is not compact. But $\tilde{X}$ is closed
in $\mathbb{R}^{n}$. By Proposition \ref{prop-Alexandrov} we get
an algebraic map $i:\mathbb{R}^{n}\to\dot{\mathbb{R}^{n}}$ where
$\dot{\mathbb{R}^{n}}$ is an affine algebaic variety which is a one
point compactification of $\mathbb{R}^{n}$, i.e. $\dot{\mathbb{R}^{n}}=i\left(\mathbb{R}^{n}\right)\cup\left\{ \infty\right\} $.
We also get that $\mathbb{R}^{n}$ and $i\left(\mathbb{R}^{n}\right)$
are algebraically isomorphic, which means they are also QN isomorphic.
Thus, $\tilde{X}$ is QN isomorphic to some locally closed subset
$i\left(\tilde{X}\right)$ of the compact variety $\dot{\mathbb{R}^{n}}$.
As $i\left(\tilde{X}\right)\cup\left\{ \infty\right\} =:\dot{\tilde{X}}$
is closed in $\dot{\mathbb{R}^{n}}$, it is compact. Now take some
$f\in W_{Z}\subset\mathcal{S}\left(\tilde{X}\right)$, and get that
$i_{*}f:=f\circ i^{-1}\in\mathcal{S}\left(i\left(\tilde{X}\right)\right)$.
By Lemma \ref{lem-ext-by-0-affine}, as $i\left(\tilde{X}\right)$
is open in $\dot{\tilde{X}}$, there exist $\dot{f}\in\mathcal{S}\left(\dot{\tilde{X}}\right)$
such that $i_{*}f=\dot{f}|_{i\left(\tilde{X}\right)}$ . Now define
$\dot{V}:=\dot{\tilde{X}}\backslash\left(i\left(Z\right)\cup\left\{ \infty\right\} \right)$.
As $i$ is a QN isomorphism, $\dot{V}$ is open in $\dot{\tilde{X}}$.
By the compact case, $Res_{\dot{V}}^{\dot{\tilde{X}}}\left(\dot{f}\right)\in\mathcal{S}\left(\dot{V}\right)$.
Note that $\dot{V}$ is QN isomorphic to $V$ by $i^{-1}|_{\dot{V}}$
. Thus, $\left(i^{-1}|_{\dot{V}}\right)_{*}Res_{\dot{V}}^{\dot{\tilde{X}}}\left(\dot{f}\right)\in\mathcal{S}\left(V\right)$.
Finally, $\left(i^{-1}|_{\dot{V}}\right)_{*}Res_{\dot{V}}^{\dot{\tilde{X}}}\left(\dot{f}\right)=\left(i^{-1}|_{\dot{V}}\right)_{*}\left(\left(i_{*}f\right)|_{\dot{V}}\right)=f|_{V}$
and thus $f|_{V}\in\mathcal{S}\left(V\right)$.
\end{proof}
\begin{thm}
\label{thm-affine-char}(Characterization of Schwartz functions on
open subset - the affine case) Let $X$ be an affine QN variety, and
let $Z\subset X$ be some closed semi-algebraic subset. Define $V:=X\backslash Z$
and $W_{Z}:=\left\{ \phi\in\mathcal{S}\left(X\right)|\phi\text{ is flat on }Z\right\} $.
Then extension by zero $Ext_{V}^{X}:\mathcal{S}\left(V\right)\to W_{Z}$
is an isomorphism of Fréchet spaces, whose inverse is $Res_{V}^{X}:W_{Z}\to\mathcal{S}\left(V\right)$.
\end{thm}

\begin{proof}
By \ref{prop-Closed-Frechet-subspace}, as $W_{Z}=\bigcap\limits _{z\in Z}\left\{ \phi\in\mathcal{S}\left(X\right)|\phi\text{ is flat on }z\right\} $
is a closed subspace of $\mathcal{S}\left(X\right)$, as an intersection
of closed subets, it is a Fréchet space. Lemma \ref{lem-ext-by-0-affine}
shows that for any $f\in\mathcal{S}\left(V\right)$, we get $Ext_{V}^{X}\left(f\right)\in\mathcal{S}\left(X\right)$
and $Ext_{V}^{X}\left(f\right)$ is flat on $Z$, i.e. $Ext_{V}^{X}\left(\mathcal{S}\left(V\right)\right)\subset W_{Z}$.
This extension is a continuous map. To show that, consider the closed
set $\tilde{X}$ corresponding to $X$. Define $W:=\mathbb{R}^{n}\backslash Z$.
This is an open semi-algebraic set, thus a Nash variety. $V=W\cap\tilde{X}$
so $V$ is closed in $W$. Take some closed embedding $W\hookrightarrow\mathbb{R}^{N}$.
Then, by Remark \ref{rem-NQN-Like-Sch-Tempered} $\mathcal{S}\left(V\right)\cong\mathcal{S}\left(W\right)/I_{Sch}^{W}\left(V\right)$.
$W$ is open in $\mathbb{R}^{n}$ so by Theorem \ref{AG-thm-Char-Nash},
$Ext_{W}^{\mathbb{R}^{n}}$ is a closed embedding, and thus continuous,
$\mathcal{S}\left(W\right)\hookrightarrow\mathcal{S}\left(\mathbb{R}^{n}\right)$.
Therefore, the map 
\[
\mathcal{S}\left(V\right)\cong\mathcal{S}\left(W\right)/I_{Sch}^{W}\left(V\right)\to\mathcal{S}\left(\mathbb{R}^{n}\right)/I_{Sch}^{\mathbb{R}^{n}}\left(\tilde{X}\right)=\mathcal{S}\left(\tilde{X}\right)
\]
is continuous, i.e. $Ext_{V}^{\tilde{X}}$ is continuous, and $Ext_{V}^{X}$
is continuous as well.

We saw in Proposition \ref{prop-res-Wz-to-U-is-Sch} that $Res_{V}^{X}\left(W_{Z}\right)\subset\mathcal{S}\left(V\right)$.

As $Res_{V}^{X}\circ Ext_{V}^{X}:\mathcal{S}\left(V\right)\to\mathcal{S}\left(V\right)$
is the identity operator by definition, and so is $Ext_{V}^{X}\circ Res_{V}^{X}:W_{Z}\to W_{Z}$,
we get that $Ext_{V}^{X}$ is a continuous bijection. By Theorem \ref{thm-Banach-open-mapping}
this means $Ext_{V}^{X}$ is an isomorphism of Fréchet spaces.
\end{proof}
\begin{cor}
\label{cor-Sch-flat-at-point}Let X be an affine QN variety. A Schwartz
function $f\in\mathcal{S}\left(X\right)$ is flat at $p\in X$ if
and only if $f|_{X\backslash\left\{ p\right\} }\in\mathcal{S}\left(X\backslash\left\{ p\right\} \right)$.
\end{cor}

\begin{proof}
apply Theorem \ref{thm-affine-char} to $Z=\left\{ p\right\} $.
\end{proof}
\begin{rem}
By the same argument for an arbitrary function $f\in C^{\infty}\left(X\right)$
(i.e. a function that is a restriction of a smooth function from an
open set in which the set corresponding to $X$ is closed) and any
$p\in X$, the following conditions are equivalent:

(1) $f$ is flat at $p$.

(2) There exists a smooth compactly supported function $\rho$ on
$\mathbb{R}^{n}$, such that $\rho$ is identically 1 on some open
neighborhood of $p$ and $\left(f\cdot\rho\right)|_{X\backslash\left\{ p\right\} }\in\mathcal{S}\left(X\backslash\left\{ p\right\} \right)$.
\end{rem}

\begin{defn}
Let $X$ be an affine QN variety. Define the \textbf{space of tempered
distributions on $\boldsymbol{X}$} as the space of continuous linear
functionals on $\mathcal{S}\left(X\right)$. Denote this space by
$\mathcal{S}^{*}\left(X\right)$.
\end{defn}

$\:$

\subsection{General QN}
\begin{defn}
\textit{\label{def-Schwartz-on-Gen-QN-Cov-C}A Schwartz function on
a (general) QN variety $X$ with cover $C$:}\textcolor{black}{{} let
$X$ be a QN variety ,and let }$C$ be an open affine QN cover of
$X$ - i.e. $\bigcup\limits _{i=1}^{m}X_{i}=X$. Denote by $Func\left(X,\mathbb{R}\right)$
the space of all real valued functions on $X$. There is a natural
map $\psi:\bigoplus\limits _{i=1}^{m}Func\left(X_{i},\mathbb{R}\right)\to Func\left(X,\mathbb{R}\right)$.
Define the space of Schwartz functions on $X$ associated with the
cover $C$ by $\mathcal{S}_{C}\left(X\right):=\psi\left(\bigoplus\limits _{i=1}^{m}\mathcal{S}\left(X_{i}\right)\right)$.
\end{defn}

\begin{lem}
\label{lem-Schwartz-Gen-Frech}Let $X$ be a QN variety, and let $C$
be an open affine QN cover of $X$. Then the space $\mathcal{S}_{C}\left(X\right)$
is a Fréchet space.
\end{lem}

\begin{proof}
First note that $\psi\left(\bigoplus\limits _{i=1}^{m}\mathcal{S}\left(X_{i}\right)\right)\cong\bigoplus\limits _{i=1}^{m}\mathcal{S}\left(X_{i}\right)/Ker\left(\psi|_{\bigoplus\limits _{i=1}^{m}\mathcal{S}\left(X_{i}\right)}\right)$
with the natural quotient topology. A direct sum of Fréchet spaces
is a Fréchet space. The kernel of $\psi|_{\bigoplus\limits _{i=1}^{m}\mathcal{S}\left(X_{i}\right)}$
is a closed subspace, as $\bigoplus\limits _{i=1}^{m}s_{i}\in Ker\left(\psi|_{\bigoplus\limits _{i=1}^{m}\mathcal{S}\left(X_{i}\right)}\right)$
if and only if for any $x\in X$, $\sum\limits _{i\in J_{x}}s_{i}\left(x\right)=0$,
where $J_{x}:=\left\{ 1\leq i\leq m|x\in X_{i}\right\} $, i.e. the
kernel is given by infinitely many \textquotedbl{}closed conditions\textquotedbl{}.
So by Proposition \ref{prop-Closed-Frechet-subspace} and {[}T - Proposition
7.9{]}, the quotient is a Fréchet space as well.
\end{proof}
\begin{lem}
\label{lem-Sch-cover-indep}Let $X$ be a QN variety and let $C,\:D$
be two open QN covers of $X$. Then $\mathcal{S}_{C}\left(X\right)\cong\mathcal{S}_{D}\left(X\right)$
as Fréchet spaces.
\end{lem}

\begin{proof}
To prove this lemma, we will show the spaces have a continuous bijective
map between them. Thus, by Theorem \ref{thm-Banach-open-mapping}
they are isomorphic as Fréchet spaces. 

We start with bijectiveness. Let the open affine cover $C$ be $X=\bigcup\limits _{i=1}^{m}Y_{i}$
and the open affine cover $D$ be $X=\bigcup\limits _{j=1}^{n}X_{j}$.
By definition, $s\in\mathcal{S}_{D}\left(X\right)\Leftrightarrow s=\sum\limits _{j=1}^{m}Ext_{X_{j}}^{X}s_{j}$
where for each $j$, $s_{j}\in\mathcal{S}\left(X_{j}\right)$. Fix
one such affine QN variety $X_{j}$. $X_{j}$ can be covered by the
open QN cover $X_{j}=\bigcup\limits _{i=1}^{n}\left(Y_{i}\cap X_{j}\right)=:\bigcup\limits _{i=1}^{n}X_{ij}$.
As $X_{j}$ is affine for any $j$, denote by $\tilde{X}_{j}\subset\mathbb{R}^{n_{j}}$
a closed set corresponding to $X_{j}$. Denote by $\tilde{X}_{ij}$
the open subsets of $\tilde{X}{}_{j}$ corresponding to $X_{ij}$.
For any $j$, $s_{j}\in\mathcal{S}\left(\tilde{X}_{j}\right)$ is
a restriction of a Schwartz function $S_{j}\in\mathcal{S}\left(U_{j}\right)$
to $\tilde{X}{}_{j}$, where $U_{j}\subset\mathbb{R}^{n_{j}}$ is
an open set in which $\tilde{X}_{j}$ is closed. As $\tilde{X}{}_{ij}\subset\tilde{X}{}_{j}$
are open semi-algebraic, we may find open semi-algebraic subsets $U_{ij}\subset U_{j}$
such that $U_{ij}\cap\tilde{X}{}_{j}=\tilde{X}{}_{ij}$. As $U_{j}$
and the $U_{ij}$'s are Nash manifolds, we may add the open Nash set
$W_{j}:=U_{j}\backslash\tilde{X}{}_{j}$ to get a Nash open cover
of $U_{j}$, and use Theorem \ref{AG-thm-part-of-unity}(partition
of unity) to get the Schwatrz functions $S_{ij}\in\mathcal{S}\left(U_{ij}\right),\:S_{W_{j}}\in\mathcal{S}\left(W_{j}\right)$,
such that extending those functions by zero sum up to $S_{j}$, i.e.
$\sum\limits _{i=1}^{n}Ext_{U_{ij}}^{U_{j}}S_{ij}+Ext_{W_{j}}^{U_{j}}S_{W_{j}}=S_{j}$.
Thus, after restricting those functions to $\tilde{X}{}_{j}$, we
can pull them back to get $\sum\limits _{i=1}^{n}Ext_{X_{ij}}^{X_{j}}s_{ij}=s_{j}$.
So $s\in\mathcal{S}_{D}\left(X\right)\Leftrightarrow s=\sum\limits _{j=1}^{m}\sum\limits _{i=1}^{n}Ext_{X_{ij}}^{X_{j}}s_{ij}$.
As both covers are finite, the sums may commute, and we get that $s\in\mathcal{S}_{D}\left(X\right)\Leftrightarrow s\in\mathcal{S}_{C}\left(X\right)$.

To prove the map is continuous, let $\psi:\bigoplus\limits _{j=1}^{n}Func\left(X_{j},\mathbb{R}\right)\to Func\left(X,\mathbb{R}\right)$
be the natural map from Definition \ref{def-Schwartz-on-Gen-QN-Cov-C}.
We begin with the claim that $\mathcal{S}\left(X_{ij}\right)\rightarrow\mathcal{S}\left(X_{j}\right)$
is a continuous map. Take the closed sets $\tilde{X}_{j}\subset\mathbb{R}^{n_{j}}$,
and $\tilde{X}_{ij}$'s as before. Each function $f_{ij}\in\mathcal{S}\left(\tilde{X}_{ij}\right)$
is a restriction of some $F_{ij}\in\mathcal{S}\left(U_{ij}\right)$,
where $U_{ij}\subset\mathbb{R}^{n_{j}}$ is an open subset such that
$U_{ij}\cap\tilde{X}_{j}=\tilde{X}_{ij}$. Denote $U_{j}:=\bigcup\limits _{i=1}^{n}U_{ij}$.
By Theorem \ref{AG-thm-Char-Nash}, the extension by zero of Schwartz
functions $Ext_{U_{ij}}^{U_{j}}F_{ij}=F_{j}\in\mathcal{S}\left(U_{j}\right)$
is a closed embedding, thus continuous. As $\tilde{X}_{j}\subset U_{j}$
is a closed subset, restricting the extension maps to $\tilde{X}_{j}$
shows that the extension by zero $\mathcal{S}\left(\tilde{X}_{ij}\right)\rightarrow\mathcal{S}\left(\tilde{X}_{j}\right)$
is continuous. Thus, by Theorem \ref{thm-Banach-open-mapping}, $\bigoplus\limits _{i}\mathcal{S}\left(X_{ij}\right)/Ker\left(\phi_{j}|_{\bigoplus\limits _{i}\mathcal{S}\left(X_{ij}\right)}\right)\cong\mathcal{S}\left(X_{j}\right)$,
where $\phi_{j}:\bigoplus\limits _{i}Func\left(X_{ij},\mathbb{R}\right)\to Func\left(X_{j},\mathbb{R}\right)$
is the natural map from Definition \ref{def-Schwartz-on-Gen-QN-Cov-C}.
Thus, we get that 
\[
\mathcal{S}_{C}\left(X\right)=\psi\left(\bigoplus\limits _{j}\mathcal{S}\left(X_{j}\right)\right)\cong\psi\left(\bigoplus\limits _{j}\phi_{j}\left(\bigoplus\limits _{i}\mathcal{S}\left(X_{ij}\right)\right)\right)\cong\phi\left(\bigoplus\limits _{j}\bigoplus\limits _{i}\mathcal{S}\left(X_{ij}\right)\right),
\]
where $\phi:\bigoplus\limits _{i,j}Func\left(X_{ij},\mathbb{R}\right)\to Func\left(X,\mathbb{R}\right)$
is the natural map. The same can be done with the cover $D$ to get
the result.
\end{proof}
In view of this lemma, we will denote the space of Schwartz functions
on a QN variety $X$ just by $\mathcal{S}\left(X\right)$ without
specifying the cover.
\begin{lem}
\label{lem-First-Result}Let $\varphi:X\to Y$ be a QN isomorphism.
Then $\varphi^{*}|_{\mathcal{S}\left(Y\right)}:\mathcal{S}\left(Y\right)\to\mathcal{S}\left(X\right)$
is an isomorphism of Fréchet spaces. 
\end{lem}

\begin{proof}
Let $f\in\mathcal{S}\left(\tilde{Y}\right)$. By Lemma \ref{lem-gen-QN-equiv}
there exist open covers $\bigcup\limits _{i=1}^{N}\tilde{X}_{i}=\tilde{X}$,
$\bigcup\limits _{i=1}^{N}\tilde{Y}_{i}=\tilde{Y}$ such that $\tilde{\varphi}|_{\tilde{X}_{i}}:\tilde{X}_{i}\xrightarrow{\sim}\tilde{Y}_{i}$
are NQN isomorphisms. As Schwartz functions form a cosheaf, every
Schwartz function on $\tilde{Y}$ is a sum of extensions of Schwartz
functions on open subsets of $\tilde{Y}$. Let those subsets be $\tilde{Y}_{i}$.
According to Lemma \ref{lem-NQN-Frechet-iso}, a Schwartz function
$s_{i}\in\mathcal{S}\left(\tilde{Y}_{i}\right)$ is pulled back to
$\varphi^{*}s_{i}\in\mathcal{S}\left(\tilde{X}_{i}\right)$, and thus,
we get $\varphi^{*}s=\sum\limits _{i=1}^{n}Ext_{\tilde{X}_{i}}^{\tilde{X}}\left(\varphi^{*}s_{i}\right)$.
By definition and Lemma \ref{lem-Sch-cover-indep} this means $\varphi^{*}s\in\mathcal{S}\left(\tilde{X}\right)$.
Together with the NQN isomorphism on each such subset, we get the
result.
\end{proof}
\begin{lem}
\label{lem-Sch-Res-to-closed}Let $X$ be a QN variety, and $Z\subset X$
be some semi-algebraic closed subset. Then $Res_{Z}^{X}\left(\mathcal{S}\left(X\right)\right)=\mathcal{S}\left(Z\right)$.
\end{lem}

\begin{proof}
Let $s\in\mathcal{S}\left(X\right)$, and let $X=\bigcup\limits _{i=1}^{m}X_{i}$
be some open affine QN cover, such that $s=\sum\limits _{i=1}^{m}Ext_{X_{i}}^{X}\left(s_{i}\right)$
for some $s_{i}\in\mathcal{S}\left(X_{i}\right)$. $Z\cap X_{i}$
is open in $Z$ and closed in $X_{i}$. By theorem \ref{thm-Res To Closed}
$s_{i}|_{Z\cap X_{i}}\in\mathcal{S}\left(Z\cap X_{i}\right)$, and
thus $s|_{Z}=\sum\limits _{i=1}^{m}Ext_{Z\cap X_{i}}^{Z}\left(s|_{Z\cap X_{i}}\right)\in\mathcal{S}\left(Z\right)$.

Now let $h\in\mathcal{S}\left(Z\right)$, and let $Z=\bigcup\limits _{i=1}^{m}\left(Z\cap X_{i}\right)$
where $X_{i}$ are as before. This is an open affine cover of $Z$
so $h=\sum\limits _{i=1}^{m}Ext_{Z\cap X_{i}}^{Z}\left(h_{i}\right)$
for some $h_{i}\in\mathcal{S}\left(Z\cap X_{i}\right)$. Each $Z\cap X_{i}$
is closed in $X_{i}$ so by theorem \ref{thm-Res To Closed} there
exist functions $H_{i}\in\mathcal{S}\left(X_{i}\right)$ such that
$H_{i}|_{Z\cap X_{i}}=h_{i}$. Thus, $\sum\limits _{i=1}^{m}Ext_{X_{i}}^{X}\left(H_{i}\right)\in\mathcal{S}\left(X\right)$.
\end{proof}
\begin{lem}
\label{lem-tempered-sheaf-affine}Let $X$ be an affine QN variety.
The assignment of the space of tempered functions to any open $V\subset X$,
together with the restriction of functions, form a sheaf on $X$.
\end{lem}

\begin{proof}
First let us show that tempered functions restricted to an open subset
remain tempered. Take a closed subset $\tilde{X}\subset\mathbb{R}^{n}$
corresponding to $X$. Tempered functions on $\tilde{X}$ are defined
as restrictions of tempered functions on some open neighborhood $U$
of $\tilde{X}$. As tempered functions on $U$ form a sheaf, and by
Remark \ref{rem-NQN-Like-Sch-Tempered}, we get that a restricted
tempered function remains tempered. It is now clear the above forms
a presheaf. The proof of the glueing property follows {[}ES - Proposition
4.3{]}. It uses again the definition of functions $f_{i}\in\mathcal{T}\left(V_{i}\right)$
on subsets $V_{i}$ of $\tilde{X}$ as restrictions of functions $\hat{f}_{i}\in\mathcal{T}\left(U_{i}\right)$
from neighborhoods $U_{i}$ of the $V_{i}$'s. Then, it uses tempered
partition of unity on the $\hat{f}_{i}$'s to create function $\hat{f}$
on $\bigcup\limits _{i}U_{i}$ such that $\hat{f}|_{\tilde{X}}=f$.
Finally, it proves that $\hat{f}|_{U_{i}}\in\mathcal{T}\left(U_{i}\right)$
for any $i$, in order to show $\hat{f}$ is tempered, what implies
$f$ is tempered.
\end{proof}
\begin{lem}
\label{lem-Tempered-eqiv}Let $X$ be a QN variety, and let $t:X\to\mathbb{R}$
be some function. Then the following conditions are equivalent:

(1) There exists an open affine QN cover $X=\bigcup\limits _{i=1}^{k}X_{i}$
such that for any $1\leq i\leq k$, $t|_{X_{i}}\in\mathcal{T}\left(X_{i}\right)$.

(2) For any open affine QN cover $X=\bigcup\limits _{i=1}^{k}X_{i}$
and any $1\leq i\leq k$, $t|_{X_{i}}\in\mathcal{T}\left(X_{i}\right)$.
\end{lem}

\begin{proof}
Clearly (2) implies (1). For the other side assume there exist two
open affine QN covers $X=\bigcup\limits _{i=1}^{k}X_{i}=\bigcup\limits _{j=k+1}^{l}X_{j}$
such that for any $k+1\leq j\leq l$, $t|_{X_{j}}\in\mathcal{T}\left(X_{j}\right)$.
Fix some $1\leq i\leq k$. Note that $\left\{ X_{i}\cap X_{j}\right\} _{j=k+1}^{l}$
is an open cover of $X_{i}$. $t|_{X_{i}\cap X_{j}}$ is a restriction
of the tempered function $t|_{X_{j}}$to the open subset $X_{i}\cap X_{j}\subset X_{i}$.
By remark \ref{rem-NQN-Like-Sch-Tempered} we get that $t|_{X_{j}}$
is a restriction to $X_{j}$ of a tempered function $T$ on an open
neighborhood $U$ in which $X_{j}$ is closed. As tempered functions
on Nash manifolds form a sheaf, take the open neighborhood $V\subset U$
of $X_{i}\cap X_{j}$ in which $X_{i}\cap X_{j}$ is closed, and get
that $\hat{t}:=T|_{V}$ is a tempered function, and so, by remark
\ref{rem-NQN-Like-Sch-Tempered} again, we get that $t|_{X_{i}\cap X_{j}}=\hat{t}|_{X_{i}\cap X_{j}}\in\mathcal{T}\left(X_{i}\cap X_{j}\right)$.
By \ref{lem-tempered-sheaf-affine}, these functions can be glued
to a unique tempered function on $X_{i}$ as they form a sheaf. Thus,
we get that $t|_{X_{i}}\in\mathcal{T}\left(X_{i}\right)$.
\end{proof}
\begin{defn}
\label{def-temp-func}Let $X$ be a QN variety. A real valued function
$t:X\to\mathbb{R}$ is called \textit{\textcolor{black}{a tempered
function on $X$}} if it satisfies the equivalent conditions of Lemma
\ref{lem-Tempered-eqiv}. Denote the space of all tempered functions
on $X$ by $\mathcal{T}\left(X\right)$.
\end{defn}

\begin{lem}
\label{lem-sXt-non-affine}Let $X$ be a QN variety, $t\in\mathcal{T}\left(X\right)$
and $s\in\mathcal{S}\left(X\right)$. Then $t\cdot s\in\mathcal{S}\left(X\right)$.
\end{lem}

\begin{proof}
Let $X=\bigcup\limits _{i=1}^{k}X_{i}$ be some open affine QN cover
such that $s=\sum\limits _{i=1}^{k}Ext_{X_{i}}^{X}\left(s_{i}\right)$
for some $s_{i}\in\mathcal{S}\left(X_{i}\right)$. Then $t|_{X_{i}}\in\mathcal{T}\left(X_{i}\right)$
and by proposition \ref{prop-sXt} $t|_{X_{i}}\cdot s_{i}\in\mathcal{S}\left(X_{i}\right)$.
Thus, $t\cdot s=\sum\limits _{i=1}^{k}Ext_{X_{i}}^{X}\left(s_{i}\cdot t|_{X_{i}}\right)\in\mathcal{S}\left(X\right)$.
\end{proof}
\begin{prop}
\label{prop-part-of-unity}(tempered partition of unity) - Let $X$
be a QN variety, and let $\left\{ V_{i}\right\} _{i=1}^{m}$ be a
finite open cover of $X$. Then:

(1) There exist tempered functions $\left\{ \alpha_{i}\right\} _{i=1}^{m}$
on $X$, such that $supp\left(\alpha_{i}\right)\subset V_{i}$ and
$\sum\limits _{i=1}^{m}\alpha_{i}=1$.

(2) We can choose $\left\{ \alpha_{i}\right\} _{i=1}^{m}$ in such
a way that for any $\varphi\in\mathcal{S}\left(X\right),\:\left(\alpha_{i}\varphi\right)|_{V_{i}}\in\mathcal{S}\left(V_{i}\right)$.
\end{prop}

The proof for the affine case is similar to that in {[}ES - Proposition
3.11{]}, with the corresponding claims here (e.g. Lemma \ref{lem-Sch-Res-to-closed}
replaces Theorem 3.7 in {[}ES{]}). The idea behind this proof is to
take the corresponding set in the Nash manifold $\mathbb{R}^{n}$,
and extend the $V_{i}$'s to some open semi-algebraic sets covering
$\mathbb{R}^{n}$. By Theorem \ref{AG-thm-part-of-unity} there is
a tempered partition of unity on those extending sets, so we can reduce
to our $V_{i}$'s to get the result. For the general case we use claims
in the Appendix as follows:
\begin{proof}
(1) By definition of $X$, there exists an open affine cover $X=\bigcup\limits _{j=1}^{n}V_{j}$.
Thus, for each $i$, there exists an open affine cover $V_{i}=\bigcup\limits _{j=1}^{n}V_{ij}$
where $V_{ij}:=V_{i}\cap V_{j}$. By Proposition \ref{prop-proper-ref}
there exists a cover of $X$ by $\bigcup\limits _{i,j}V_{ij}'=X$
where $V'_{ij}\subset\bar{V}_{ij}'\subset V_{ij}$. By Corollary \ref{cor-basis-Mf},
there exist a finite collection of continuous functions $G_{ijk}:V_{ij}\rightarrow\mathbb{R}$
and open sets $V'_{ijk}$ such that $V'_{ijk}=\left\{ x\in V_{ij}|G_{ijk}\left(x\right)\neq0\right\} ,\:G_{ijk}|_{V'_{ijk}}$
is positive and QN and $\bigcup\limits _{k}V'_{ijk}=V'_{ij}$. It
gives a finite cover of $X$ which is a refinement of $U_{i}$. In
order to have a unified system of indices we denote $V_{ijk}:=V_{ij}$.
We re-index it to one index cover $V_{l}$. By the same re-indexation
we get $G_{l}$ and $V'_{l}$ . Extend $G_{l}$ by zero to a function
$\widetilde{G_{l}}$ on $X$. It is continuous. Denote $G=\left(\sum\widetilde{G_{l}}\right)/\left(2n\right)$
where $n$ is the number of values of the index $l$. Consider $G|_{V_{l}}$
. This is a strictly positive continuous semi-algebraic function on
an affine QN variety. Lemma \ref{lem-majoration} shows that continuous
strictly positive semi-algebraic function on an affine QN variety
can be bounded from below by a strictly positive QN function. Thus,
$G|_{V_{l}}$ can be bounded from below by a strictly positive QN
function $g'_{l}$. Denote $H_{l}:=G_{l}/g'_{l}$. Extending $H_{l}$
by zero outside $V_{l}$ to $X$ we obtain a collection of continuous
semi-algebraic functions $F_{l}$. Note that $F_{l}$ is not smooth.
It is easy to see that $X_{F_{l}}$ is a refinement of $V_{i}$.

Now, let $\rho:\mathbb{R}\rightarrow[0,1]$ be a smooth function such
that 
\[
\rho\left((-\infty,0.1]\right)=\left\{ 0\right\} ,\:\rho\left([1,\infty)\right)=\left\{ 1\right\} 
\]
Denote $\beta_{l}:=\rho\circ F_{l}$ and $\gamma_{l}=\frac{\beta_{l}}{\sum\beta_{l}}$.
It is easy to see that $\gamma_{l}$ are tempered. For every $l$
we choose $i\left(l\right)$ such that $X_{F_{l}}\subset U_{i\left(l\right)}$.
Define $\alpha_{i}:=\sum\limits _{l|i\left(l\right)=i}\gamma_{l}$.
It is easy to see that $\alpha_{i}$ is a tempered partition of unity.

(2) Take some $\varphi\in\mathcal{S}\left(X\right)$. By definition
$\varphi=\sum\limits _{j=1}^{n}Ext_{X_{j}}^{X}\varphi_{j}$ where
the $X_{j}$ 's are the affine open cover of $X$ and $\varphi_{j}\in\mathcal{S}\left(X_{j}\right)$.
By definition, $\alpha_{i}|_{X_{j}}\in\mathcal{T}\left(X_{j}\right)$,
and by Lemma \ref{prop-sXt} $\varphi_{j}\cdot\alpha_{i}|_{X_{j}}\in\mathcal{S}\left(X_{j}\right)$.
By part (1), $supp\left(\alpha_{i}\right)\subset V_{i}$ so $supp\left(\varphi_{j}\cdot\alpha_{i}|_{X_{j}}\right)\subset V_{i}\cap X_{j}$
and $\varphi_{j}\cdot\alpha_{i}|_{X_{j}}$ is flat on $X_{j}\backslash V_{i}$.
By Theorm \ref{thm-affine-char}, $\left(\varphi_{j}\cdot\alpha_{i}|_{X_{j}}\right)|_{V_{i}\cap X_{j}}\in\mathcal{S}\left(V_{i}\cap X_{j}\right)$.
Note that $V_{i}\cap X_{j}$ is an open affine QN cover of $V_{i}$,
so 
\[
\sum\limits _{j=1}^{n}Ext_{V_{i}\cap X_{j}}^{V_{i}}\left(\varphi_{j}\cdot\alpha_{i}|_{X_{j}}\right)|_{V_{i}\cap X_{j}}\in\mathcal{S}\left(V_{i}\right).
\]
Finally, by the definition of the $\varphi_{j}$'s we get that $\sum\limits _{j=1}^{n}Ext_{V_{i}\cap X_{j}}^{V_{i}}\left(\varphi_{j}\cdot\alpha_{i}|_{X_{j}}\right)|_{V_{i}\cap X_{j}}=\left(\varphi\cdot\alpha_{i}\right)|_{V_{i}}$.
\end{proof}
\begin{rem}
\label{rem-temp-Sch-is-Sch}Actually, what we proved in part (2) is
that for any tempered function $\beta\in\mathcal{T}\left(X\right)$
whose support is a subset of some open subset $U$, i.e. $supp\left(\beta\right)\subset U,\:U\subset X$,
and for any $\varphi\in\mathcal{S}\left(X\right)$ we get $\left(\beta\cdot\varphi\right)|_{U}\in\mathcal{S}\left(U\right)$.
\end{rem}

\begin{defn}
Let $X$ be a QN variety. A function $f:X\to\mathbb{R}$ is called
\textbf{flat} at $p\in X$ if there exists an open affine QN subset
$U\subset X$ such that $p\in U$ and $f|_{U}$ is flat at $p$. If
$f$ is flat at all $p\in Z$ for some $Z\subset X$, $f$ is called
\textbf{flat} at $Z$.
\end{defn}

\begin{prop}
\label{prop-ext-by-0}(extension by zero for non affine varieties).
Let $X$ be a QN variety, and $U$ an open subset of $X$. Then the
extension by zero to $X$ of a Schwartz function on $U$ is a Schwartz
function on $X$, which is flat at $X\backslash U$.
\begin{proof}
Take an open affine QN cover $X=\bigcup\limits _{i=1}^{k}X_{i}$.
Then $U=\bigcup\limits _{i=1}^{k}\left(U\cap X_{i}\right)$ is an
open affine QN cover of $U$, what makes $U$ a QN variety as well.
Take some $s\in\mathcal{S}\left(U\right)$. So $s=\sum\limits _{i=1}^{k}Ext_{U\cap X_{i}}^{U}\left(s_{i}\right)$
for some $s_{i}\in\mathcal{S}\left(U\cap X_{i}\right)$. The set $U_{i}:=U\cap X_{i}$
is open in $X_{i}$, so take the closed corresponding set $\tilde{X}$
in $\mathbb{R}^{n_{i}}$, and its corresponding open subset $\tilde{U}_{i}$.
Then there exists an open semi-algebraic set $V_{i}\subset\mathbb{R}^{n}$
such that $\tilde{U}_{i}=V_{i}\cap\tilde{X}_{i}$. By remark \ref{rem-NQN-Like-Sch-Tempered},
$s_{i}$ corresponds to some $\tilde{s}_{i}=S_{i}|_{\tilde{U}_{i}}$where
$S_{i}\in\mathcal{S}\left(V_{i}\right)$. By Theorem \ref{AG-thm-Char-Nash},
$S_{i}$ can be extended by zero to the whole of $\mathbb{R}^{n_{i}}$
and this extension is flat outside $V_{i}$, and in particular in
$\tilde{X}_{i}\backslash\tilde{U}_{i}$. This extension can be reduced
to $\tilde{X}_{i}$ to make a Schwartz function on $\tilde{X}_{i}$
which extends $\tilde{s}_{i}$ to $\tilde{X}_{i}$ by zero, and thus
is flat on $\tilde{X}_{i}\backslash\tilde{U}_{i}$. Thus, $s_{i}$
can be extended to a Schwartz function on $X_{i}$ which is flat on
$X_{i}\backslash U_{i}$ . So $Ext_{U}^{X}\left(s\right)=Ext_{U}^{X}\left(\sum\limits _{i=1}^{k}Ext_{U_{i}}^{U}\left(s_{i}\right)\right)=\sum\limits _{i=1}^{k}Ext_{U}^{X}\left(Ext_{U_{i}}^{U}\left(s_{i}\right)\right)=\sum\limits _{i=1}^{k}Ext_{X_{i}}^{X}\left(Ext_{U_{i}}^{X_{i}}\left(s_{i}\right)\right)$
which, by definition, is a Schwartz function on $X$ which is flat
on $X\backslash U$.
\end{proof}
\end{prop}

\begin{thm}
\label{thm-general-char}(Characterization of Schwartz functions on
open subset - the general case) Let $X$ be a QN variety, and let
$Z\subset X$ be some closed subset. Define $U:=X\backslash Z$ and
$W_{Z}:=\left\{ \phi\in\mathcal{S}\left(X\right)|\phi\text{ is flat on }Z\right\} $.
Then $W_{Z}$ is a closed subspace of $\mathcal{S}\left(X\right)$,
and thus it is a Fréchet space. Furthermore, extension by zero $Ext_{U}^{X}:\mathcal{S}\left(U\right)\to W_{Z}$
is an isomorphism of Fréchet spaces, whose inverse is $Res_{U}^{X}:W_{Z}\to\mathcal{S}\left(U\right)$.
\end{thm}

\begin{proof}
As for the first part, $W_{Z}=\bigcap\limits _{z\in Z}\left\{ \phi\in\mathcal{S}\left(X\right)|\phi\text{ is flat on }z\right\} $
is an intersection of closed sets, it is a closed subspace of $\mathcal{S}\left(X\right)$,
and thus a Fréchet space. For the second part, by Proposition \ref{prop-ext-by-0}
the extension of a function in $\mathcal{S}\left(U\right)$ by zero
to $X$ is a function in $\mathcal{S}\left(X\right)$ that is flat
at $Z$, i.e. $Ext_{U}^{X}\left(\mathcal{S}\left(U\right)\right)\subset W_{Z}$.
Furthermore, we will claim further on that $Ext_{U}^{X}$ is continuous.
Before that, we will show the opposite direction - that $Res_{U}^{X}\left(W_{Z}\right)\subset\mathcal{S}\left(U\right)$.
Let $f\in W_{Z}$. We want to show $f|_{U}=\sum\limits _{i=1}^{n}Ext_{U_{i}}^{U}s_{i}$
where $X=\bigcup\limits _{i=1}^{n}X_{i}$ is an open affine QN cover
of $X$, $U_{i}:=U\cap X_{i}$, (this makes $U=\bigcup\limits _{i=1}^{n}U_{i}$
an open affine cover of $U$) and $s_{i}\in\mathcal{S}\left(U_{i}\right)$.
Note that $s_{i}\neq f|_{U_{i}}$. As $f\in\mathcal{S}\left(X\right)$,
and as $X=\bigcup\limits _{i=1}^{n}X_{i}$ is an affine QN open cover
of $X$, we can use Proposition \ref{prop-part-of-unity} to get tempered
functions $\alpha_{i}$ such that $\alpha_{i}\cdot f\in\mathcal{S}\left(X_{i}\right)$.
As $f\in W_{Z}$ is flat at any $x\in Z$, $\alpha_{i}\cdot f$ is
flat at any $x\in Z_{i}:=Z\cap X_{i}$. Thus, $\alpha_{i}\cdot f\in W_{Z_{i}}$,
and by Theorem \ref{thm-affine-char} we get $\alpha_{i}\cdot f\in\mathcal{S}\left(U_{i}\right)$
where $U_{i}:=U\cap X_{i}$. $\bigcup\limits _{i=1}^{n}U_{i}$ is
an open affine QN cover of $U$, so $\sum\limits _{i=1}^{n}Ext_{U_{i}}^{U}\left(\alpha_{i}\cdot f\right)\in\mathcal{S}\left(U\right)$.
Furthermore, by the definition of the $\alpha$'s, $\sum\limits _{i=1}^{n}Ext_{U_{i}}^{U}\left(\alpha_{i}\cdot f\right)=f|_{U}$.\textcolor{red}{{} }

By Theorem \ref{thm-affine-char}, for any $i=1,...,n$ we have the
continuous map $Ext_{U_{i}}^{X_{i}}:\mathcal{S}\left(U_{i}\right)\rightarrow\mathcal{S}\left(X_{i}\right)$.
As $n<\infty$, we get the continuous map $Ext_{U}^{X}:\bigoplus\limits _{i=1}^{n}\mathcal{S}\left(U_{i}\right)\rightarrow\bigoplus\limits _{i=1}^{n}\mathcal{S}\left(X_{i}\right)$.
Recall that $\mathcal{S}\left(X\right):=\psi\left(\bigoplus\limits _{i=1}^{n}\mathcal{S}\left(X_{i}\right)\right)\cong\bigoplus\limits _{i=1}^{n}\mathcal{S}\left(X_{i}\right)/Ker\left(\psi|_{\bigoplus\limits _{i=1}^{n}\mathcal{S}\left(X_{i}\right)}\right)$
and get the continuous map $Ext_{U}^{X}:\mathcal{S}\left(U\right)\rightarrow\mathcal{S}\left(X\right)$.
As this map is bijective, we get by the Theorem \ref{thm-Banach-open-mapping}
$Ext_{U}^{X}$ is an isomorphism of Fréchet spaces.
\end{proof}
\begin{defn}
Let $X$ be a QN variety. Define the \textbf{space of tempered distributions
on $\boldsymbol{X}$} as the space of continuous linear functionals
on $\mathcal{S}\left(X\right)$. Denote this space by $\mathcal{S}^{*}\left(X\right)$.
\end{defn}

\begin{thm}
\label{thm-temp-dist-onto-general}Let $X$ be a QN variety, and let
$U\subset X$ be some semi-algebraic open subset. Then $Ext_{U}^{X}:\mathcal{S}\left(U\right)\hookrightarrow\mathcal{S}\left(X\right)$
is a closed embedding, and the restriction morphism $\mathcal{S}^{*}\left(X\right)\rightarrow\mathcal{S}^{*}\left(U\right)$
is onto.
\end{thm}

\begin{proof}
As $W_{Z}=\bigcap\limits _{z\in Z}\left\{ \phi\in\mathcal{S}\left(X\right)|\phi\text{ is flat on }z\right\} $
is an intersection of closed subsets, it is a closed subspace of $\mathcal{S}\left(X\right)$
and thus, by Theorem \ref{thm-general-char} the first part is proved.
The second part follows from the fact that $\mathcal{S}\left(X\right)$
is a Fréchet space and from Theorem \ref{thm-Hahn-Banach}.
\end{proof}

\section{\textbf{Sheaves and Cosheaves}}

The aim of this section is to prove that tempered functions and tempered
distributions form sheaves and that Schwartz functions form a cosheaf.
Unlike the algebraic case, this can be done to both the affine and
the general case. The proofs for the affine cases, and for the general
case of tempered functions, are the same as the those in the {[}ES{]}
regarding algebraic varieties. Thus we give short sketches of the
proofs together with relevant statements in this paper replacing statements
in {[}ES{]}.

First, let us recall Lemma \ref{lem-tempered-sheaf-affine}:
\begin{lem}
Let $X$ be an affine QN variety. The assignment of the space of tempered
functions to any open $V\subset X$, together with the restriction
of functions, form a sheaf on $X$.
\end{lem}

\begin{cor}
\label{cor-tempered-sheaf}Let $X$ be a QN variety. The assignment
of the space of tempered functions to any open $U\subset X$, together
with the restriction of functions, form a sheaf on $X$.
\end{cor}

\begin{proof}
The proof follows the proof of {[}ES - Proposition 5.11{]}. By the
definition of tempered functions on QN varieties and by Lemma \ref{lem-tempered-sheaf-affine},
they form a presheaf. Using induction on the number of the covering
open subsets, it is enough to show the following: 

Let $X$ be a QN variety and let $U_{1}\cup U_{2}=X$ be an open cover
of $X$. Assume we are given $t_{i}\in\mathcal{T}\left(U_{i}\right)$
such that \textcolor{black}{$t_{1}|_{U_{1}\cap U_{2}}=t_{2}|_{U_{1}\cap U_{2}}$.
Then, there exists a unique function $t\in\mathcal{T}\left(U_{1}\cap U_{2}\right)$
such that $t|_{U_{i}}=t_{i}$. }

The existence of a function $t:U_{1}\cup U_{2}\to\mathbb{R}$ such
that $t|_{U_{i}}=t_{i}$ is clear. We shall now show it is tempered.
Consider some affine open cover $X=\bigcup\limits _{j=1}^{k}X_{j}$.
Then $U_{i}=\bigcup\limits _{j=1}^{k}\left(U_{i}\cap X_{j}\right)$
is an affine open cover of $U_{i}$, and $U_{1}\cup U_{2}=\bigcup\limits _{j=1}^{k}\left(\left(U_{1}\cup U_{2}\right)\cap X_{j}\right)$
is an affine open cover of $U_{1}\cup U_{2}$. As $t_{i}\in\mathcal{T}\left(U_{i}\right)$,
we get that $t_{i}|_{U_{i}\cap X_{j}}\in\mathcal{T}\left(U_{i}\cap X_{j}\right)$.
As $\left(U_{1}\cup U_{2}\right)\cap X_{j}$ is affine, and $\bigcup\limits _{i=1}^{2}\left(U_{i}\cap X_{j}\right)$
is an affine open cover of it, and as \textcolor{black}{$t_{1}|_{U_{1}\cap U_{2}}=t_{2}|_{U_{1}\cap U_{2}}$,
we get by Lemma \ref{lem-tempered-sheaf-affine} that $t|_{\left(U_{1}\cup U_{2}\right)\cap X_{j}}\in\mathcal{T}\left(\left(U_{1}\cup U_{2}\right)\cap X_{j}\right)$.
Thus $t\in\mathcal{T}\left(U_{1}\cup U_{2}\right)$.}
\end{proof}
Before dealing with cosheaves in the category of real vector spaces,
let us recall their definition:

First, define the category $Top(X)$ to be such that its objects are
the open sets of $X$, and its morphisms are the inclusion maps. A
\emph{pre-cosheaf} $F$ on a topological space $X$ is a covariant
functor from $Top(X)$ to the category of real vector spaces. A \emph{cosheaf}
on a topological space $X$ is a pre-cosheaf, such that for any open
$V\subset X$ and any open cover $\{V_{i}\}_{i\in I}$ of $V$, the
following sequence is exact:
\[
\bigoplus\limits _{\left(i,j\right)\in I^{2}}F\left(V_{i}\cap V_{j}\right)\xrightarrow{Ext_{1}}\bigoplus\limits _{i\in I}F\left(V_{i}\right)\xrightarrow{Ext_{2}}F\left(V\right)\xrightarrow{}0,
\]

where the $k$-th coordinate of $Ext_{1}(\bigoplus\limits _{(i,j)\in I^{2}}\xi_{i,j})$
is $\sum\limits _{i\in I}Ext_{V_{k}\cap V_{i}}^{V_{k}}(\xi_{k,i}-\xi_{i,k})$,
and $Ext_{2}(\bigoplus\limits _{i\in I}\xi_{i}):=\sum\limits _{i\in I}Ext_{V_{i}}^{V}(\xi_{i})$.
When exactness will be proven in Proposition \ref{Schwartz-is-a-cosheaf}
below all calculations will be quickly reduced to finite subcovers.
A cosheaf is \emph{flabby} if for any two open subsets $U,V\subset X$
such that $V\subset U$, the morphism $Ext_{V}^{U}:F(V)\to F(U)$
is injective.
\begin{lem}
\label{lem-Sch-cosheaf-affine}Let $X$ be an affine QN variety. The
assignment of the space of Schwartz functions to any open $U\subset X$,
together with the extension by zero $Ext_{U}^{V}$ from $U$ to any
other open $V\supset U$, form a flabby cosheaf on X.
\end{lem}

\begin{proof}
The proof follows {[}ES - Proposition 4.5{]} - By extension by 0 (see
Theorem \ref{thm-affine-char}), $X$ is a pre-cosheaf. Now we shall
prove the exactness:

The $\bigoplus\limits _{j=1}^{l}F\left(U_{j}\right)\longrightarrow F\left(U\right)\longrightarrow0$
part follows immediately from partition of unity. The second part
uses induction on the number of the covering sets. The base step uses
the fact that the two functions sum to zero everywhere, and that their
extensions are flat outside the sets' intersection. Thus, by the characterization
property given in Theorem \ref{thm-affine-char}, their restriction
to the intersection is Schwartz. In the inductive step, we use partition
of unity on $k+1$ sets, to create some new Schwartz functions on
the first $k$ sets whose extensions sum to zero. We then define and
prove the claim for the $k+1$'th function, using the fact that each
of the $k$ functions is flat at each point they are solely defined
. 

\textcolor{black}{The precise proof is the same as in {[}ES{]}, with
the following replacements: }

\textcolor{black}{Theorem \ref{thm-temp-dist-onto-general} replaces
{[}ES - Theorem 3.25{]}, }

\textcolor{black}{Remark \ref{Rem-NQN-Like} replaces {[}ES - Cor.
3.10{]}, }

\textcolor{black}{Proposition \ref{prop-part-of-unity} replaces {[}ES
- Prop. 3.11{]}, }

\textcolor{black}{Proposition \ref{prop-ext-by-0}replaces {[}ES -
Prop. 3.16{]}, }

\textcolor{black}{Theorem \ref{thm-affine-char} replaces {[}ES -
Thm. 3.20{]}, }

\textcolor{black}{Remark \ref{rem-temp-Sch-is-Sch} replaces {[}ES
- Cor. 3.22{]}, }

\textcolor{black}{Corollary \ref{cor-Sch-flat-at-point} replaces
{[}ES - Cor. 3.21{]}, }

\textcolor{black}{Propostion \ref{lem-tempered-sheaf-affine} replaces
{[}ES - Prop. 4.3{]}, and }

\textcolor{black}{Lemma \ref{lem-sXt-non-affine} replaces {[}ES -
Prop. 3.9{]}.}
\end{proof}
\begin{cor}
\label{cor-Sch-cosheaf}Let $X$ be a QN variety. The assignment of
the space of Schwartz functions to any open $U\subset X$, together
with the extension by zero $Ext_{U}^{V}$ from $U$ to any other open
$V\supset U$, form a flabby cosheaf on X.
\end{cor}

\begin{proof}
By Proposition \ref{prop-ext-by-0}, the Schwartz functions form a
pre-cosheaf on $X$. 

Let $\bigcup\limits _{i=1}^{k}X^{i}=X$ be a finite open cover such
that for any $i$, $X^{i}$ is affine. Let $U\subset X$ be some open
set, $\bigcup\limits _{j=1}^{l}U_{j}=U$ is some open cover, and let
$s\in\mathcal{S}\left(U\right)$. Then for any $i$, $U^{i}:=U\cap X^{i}$
is an open subset of the affine $X^{i}$, and $\bigcup\limits _{j=1}^{l}U_{j}^{i}=U^{i}$
where $U_{j}^{i}:=U^{i}\cap U_{j}$ is an open cover of $U^{i}$.
By definition, $s=\sum\limits _{i=1}^{k}Ext_{U^{i}}^{U}\left(s^{i}\right)$
where $s^{i}\in\mathcal{S}\left(U^{i}\right)$. By Lemma \ref{lem-Sch-cosheaf-affine}
we know that $s^{i}=\sum\limits _{j=1}^{l}Ext_{U_{j}^{i}}^{U^{i}}\left(s_{j}^{i}\right)$
where $s_{j}^{i}\in\mathcal{S}\left(U_{j}^{i}\right)$. So we get
that $s=\sum\limits _{i=1}^{k}Ext_{U^{i}}^{U}\sum\limits _{j=1}^{l}Ext_{U_{j}^{i}}^{U^{i}}s_{j}^{i}$
and as the sums are finite, we may write $s=\sum\limits _{j=1}^{l}Ext_{U_{j}}^{U}\sum\limits _{i=1}^{k}Ext_{U_{j}^{i}}^{U_{j}}s_{j}^{i}=\sum\limits _{j=1}^{l}Ext_{U_{j}}^{U}s_{j}$
where $s_{j}:=\sum\limits _{i=1}^{k}Ext_{U_{j}^{i}}^{U_{j}}s_{j}^{i}$
and $s_{j}\in\mathcal{S}\left(U_{j}\right)$ by definition. This proves
the part 
\[
\bigoplus\limits _{j=1}^{l}F\left(U_{j}\right)\longrightarrow F\left(U\right)\longrightarrow0
\]
 of the cosheaf definition.

Now let us have $\left\{ s_{j}\right\} \subset\mathcal{S}\left(U_{j}\right)$
such that $\sum\limits _{j=1}^{l}Ext_{U_{j}}^{U}s_{j}=0$. As $\bigcup\limits _{i=1}^{k}U^{i}=U$
is an open (affine) cover of $U$, we may use Proposition \ref{prop-part-of-unity}
to get some tempered functions $\left\{ \alpha_{i}\right\} _{i=1}^{k}$
on $U$, such that $supp\left(\alpha_{i}\right)\subset U^{i}$ and
$\sum\limits _{i=1}^{k}\alpha_{i}=1$, and for any $\varphi\in\mathcal{S}\left(U\right),\:\left(\alpha_{i}\varphi\right)|_{U^{i}}\in\mathcal{S}\left(U^{i}\right)$.
Now, $S_{j}:=Ext_{U_{j}}^{U}s_{j}\in\mathcal{S}\left(U\right)$ by
Proposition \ref{prop-ext-by-0}, so $\left(\alpha_{i}\cdot S_{j}\right)|_{U^{i}}\in\mathcal{S}\left(U^{i}\right)$.
We get an affine QN variety $U^{i}$ with $\sum\limits _{j=1}^{l}\left(\left(\alpha_{i}\cdot S_{j}\right)|_{U^{i}}\right)=0$.
Notice that for each $j$, the function $\alpha_{i}\cdot S_{j}$ is
flat on $U^{i}\backslash U_{j}$, so we may use Theorem \ref{thm-affine-char}
to get $\left(\alpha_{i}\cdot S_{j}\right)|_{U_{j}^{i}}\in\mathcal{S}\left(U_{j}^{i}\right)$.
Thus, we have $\sum\limits _{j=1}^{l}Ext_{U_{j}^{i}}^{U^{i}}\left(\left(\alpha_{i}\cdot S_{j}\right)|_{U_{j}^{i}}\right)=0$
and by Lemma \ref{lem-Sch-cosheaf-affine}, we know there exists $s_{jm}^{i}\in\mathcal{S}\left(U_{j}^{i}\cap U_{m}^{i}\right)$
satisfying $\left(\alpha_{i}\cdot S_{j}\right)|_{U_{j}^{i}}=\sum\limits _{m<j}Ext_{U_{j}^{i}\cap U_{m}^{i}}^{U_{j}^{i}}\left(s_{jm}^{i}\right)-\sum\limits _{m>j}Ext_{U_{j}^{i}\cap U_{m}^{i}}^{U_{j}^{i}}\left(s_{mj}^{i}\right)$.
Sum the extensions by zero to $U_{j}$ of both sides to get 

\[
\sum\limits _{i=1}^{k}Ext_{U_{j}^{i}}^{U_{j}}\left(\left(\alpha_{i}\cdot S_{j}\right)|_{U_{j}^{i}}\right)=
\]
\[
\sum\limits _{i=1}^{k}Ext_{U_{j}^{i}}^{U_{j}}\left(\sum\limits _{m<j}Ext_{U_{j}^{i}\cap U_{m}^{i}}^{U_{j}^{i}}\left(s_{jm}^{i}\right)-\sum\limits _{m>j}Ext_{U_{j}^{i}\cap U_{m}^{i}}^{U_{j}^{i}}\left(s_{mj}^{i}\right)\right),
\]

$\:$

but

$\:$
\[
\sum\limits _{i=1}^{k}Ext_{U_{j}^{i}}^{U_{j}}\left(\left(\alpha_{i}\cdot S_{j}\right)|_{U_{j}^{i}}\right)=S_{j}|_{U_{j}}=s_{j}
\]

$\:$

and

$\:$
\[
\sum\limits _{i=1}^{k}Ext_{U_{j}^{i}}^{U_{j}}\left(\sum\limits _{m<j}Ext_{U_{j}^{i}\cap U_{m}^{i}}^{U_{j}^{i}}\left(s_{jm}^{i}\right)-\sum\limits _{m>j}Ext_{U_{j}^{i}\cap U_{m}^{i}}^{U_{j}^{i}}\left(s_{mj}^{i}\right)\right)=
\]
\[
\sum\limits _{m<j}Ext{}_{U_{j}\cap U_{m}}^{U_{j}}\sum\limits _{i=1}^{k}Ext_{U_{j}^{i}\cap U_{m}^{i}}^{U_{j}\cap U_{m}}\left(s_{jm}^{i}\right)-\sum\limits _{m>j}Ext{}_{U_{j}\cap U_{m}}^{U_{j}}\sum\limits _{i=1}^{k}Ext_{U_{j}^{i}\cap U_{m}^{i}}^{U_{j}\cap U_{m}}\left(s_{mj}^{i}\right)=
\]
\[
\sum\limits _{m<j}Ext{}_{U_{j}\cap U_{m}}^{U_{j}}\left(s_{jm}\right)-\sum\limits _{m>j}Ext{}_{U_{j}\cap U_{m}}^{U_{j}}\left(s_{mj}\right)
\]

$\:$

where $s_{jm}\in\mathcal{S}\left(U_{j}\cap U_{m}\right)$ by definition.
This sums up to 

$\:$
\[
s_{j}=\sum\limits _{m<j}Ext{}_{U_{j}\cap U_{m}}^{U_{j}}\left(s_{jm}\right)-\sum\limits _{m>j}Ext{}_{U_{j}\cap U_{m}}^{U_{j}}\left(s_{mj}\right)
\]

$\:$

what proves the part

$\:$
\[
\bigoplus\limits _{j>m}F\left(U_{j}\cap U_{m}\right)\longrightarrow\bigoplus\limits _{j=1}^{l}F\left(U_{j}\right)\longrightarrow F\left(U\right)
\]

$\:$

of the cosheaf definition.
\end{proof}
\begin{lem}
\label{lem-temp-dist-sheaf-affine}Let $X$ be an affine QN variety.
The assignment of the space of tempered distributions to any open
$U\subset X$, together with restrictions of functionals from $\mathcal{S}^{*}\left(U\right)$
to $\mathcal{S}^{*}\left(V\right)$, for any other open $V\subset U$,
form a flabby sheaf on $X$.
\end{lem}

\begin{proof}
The proof follows {[}ES - Proposition 4.4{]}. The proof uses extension
by zero to show presheaf structure. To show uniqueness, we use partition
of unity which enables us to write $s\in\mathcal{S}\left(U\right)$
as $s=\sum\limits _{i=1}^{k}\left(\beta_{i}\cdot s\right)$ where
$supp\left(\beta_{i}\right)\subset U_{i}$. Then, as $\left(\beta_{i}\cdot s\right)|_{U_{i}}\in\mathcal{S}\left(U_{i}\right)$,
and as the functionals $\xi,\:\zeta\in\mathcal{S}^{*}\left(U\right)$
agree on each subset and by the linearity of the functionals, we get
$\xi\left(s\right)-\zeta\left(s\right)=\xi\left(\sum\limits _{i=1}^{k}\left(\beta_{i}\cdot s\right)\right)-\zeta\left(\sum\limits _{i=1}^{k}\left(\beta_{i}\cdot s\right)\right)=\sum\limits _{i=1}^{k}\left(\xi\left(\beta_{i}\cdot s\right)-\zeta\left(\beta_{i}\cdot s\right)\right)=0$.
This means the uniqueness is achieved. 

The existence is proven by partition of unity on $U=\bigcup\limits _{i=1}^{k}U_{i}$,
and defining a functional $\xi\in\mathcal{S}^{*}\left(U\right)$ in
the following way: $\xi\left(s\right)=\xi\left(\sum\limits _{i=1}^{k}\left(\beta_{i}\cdot s\right)\right):=\sum\limits _{i=1}^{k}\xi_{i}\left(\beta_{i}\cdot s\right)$
where $\xi_{i}\in\mathcal{S}^{*}\left(U_{i}\right)$, $s\in\mathcal{S}\left(U\right)$
and $\beta_{i}$ are the tempered functions obtained by the partition
of unity. For any $U_{\alpha}\subset U$ open, where $\alpha\in\left\{ 1,...,k\right\} $
and $s_{\alpha}\in\mathcal{S}\left(U_{\alpha}\right)$, we get $\left(\beta_{i}|_{U_{\alpha}}\cdot s_{\alpha}\right)|_{U_{\alpha}\cap U_{i}}\in\mathcal{S}\left(U_{\alpha}\cap U_{i}\right)$
by Lemma \ref{lem-tempered-sheaf-affine} and Proposition \ref{lem-sXt-non-affine}.
As $\xi_{\alpha}|_{\mathcal{S}\left(U_{\alpha}\cap U_{i}\right)}=\xi_{i}|_{\mathcal{S}\left(U_{\alpha}\cap U_{i}\right)}$
, and $s_{\alpha}$ may be extended to a Schwartz function on $U$,
we get $\xi_{\alpha}\left(s_{\alpha}\right)=\xi_{\alpha}\left(\sum\limits _{i=1}^{k}\left(\beta_{i}\cdot s_{\alpha}\right)\right)=\sum\limits _{i=1}^{k}\xi_{\alpha}\left(\beta_{i}\cdot s_{\alpha}\right)=\sum\limits _{i=1}^{k}\xi_{i}\left(\beta_{i}\cdot s_{\alpha}\right)=\xi\left(s_{\alpha}\right)$.
This means that for any $\alpha,\:\xi|_{\mathcal{S}\left(U_{\alpha}\right)}=\xi_{\alpha}$
and the existence holds.

\textcolor{black}{Proposition \ref{prop-ext-by-0} replaces {[}ES
- Prop. 3.16{]}, Proposition \ref{prop-part-of-unity} replaces {[}ES
- Prop. 3.11{]}, and Remark \ref{rem-temp-Sch-is-Sch} replaces {[}ES
- Cor. 3.22{]}.}
\end{proof}
\begin{cor}
\label{cor-Dist-sheaf}Let $X$ be a QN variety. The assignment of
the space of tempered distributions to any open $U\subset X$, together
with restrictions of functionals from $\mathcal{S}^{*}\left(U\right)$
to $\mathcal{S}^{*}\left(V\right)$, for any other open $V\subset U$,
form a flabby sheaf on $X$.
\end{cor}

\begin{proof}
The claim that tempered distributions form a sheaf is dual to the
claim that Scwartz functions form a cosheaf, which we proved in Corollary
\ref{cor-Sch-cosheaf}.
\end{proof}

\section{Schwartz, Tempered and Tempered ``Distributions'' Over QN Vector
Bundles}

We begin this section with definitions of QN bundles and their sections.
Most of the definitions and results in this chapter follow {[}AG{]}
with light adjustments to our category.
\begin{defn}
Let $\pi:X\to B$ be a morphism of QN varieties. It is called a \textbf{QN
locally trivial fibration} with fiber $Z$ if the following holds: 
\end{defn}

\begin{itemize}
\item $Z$ is a QN variety.
\item There exists a \textit{finite} cover $B=\bigcup\limits _{i=1}^{n}U_{i}$
by open QN sets and QN isomorphisms $\nu_{i}:\pi^{-1}\left(U_{i}\right)\tilde{\to}U_{i}\times Z$
such that $\pi\circ\nu_{i}^{-1}$ is the natural projection.
\end{itemize}
\begin{defn}
Let $X$ be a QN variety. A \textbf{QN vector bundle $E$ over $X$}
is a QN locally trivial fibration with linear fiber and such that
the trivialization maps $\nu_{i}$ are fiberwise linear. By abuse
of notation, we use the same letters to denote bundles and their total
spaces.
\end{defn}

$\:$
\begin{defn}
Let $X$ be a QN variety and $E$ a QN bundle over $X$. A \textbf{QN
section} of $E$ is a section of $E$ which is a QN morphism.
\end{defn}

Now we can use the above definitions to define the more specific QN
bundles relevant to our work.
\begin{defn}
Let $X$ be a QN variety, and $E$ be a QN bundle over it. Let $X=\bigcup\limits _{i=1}^{k}X_{i}$
be an affine QN trivialization of $E$. A global section $s$ of $E$
over $X$ is called \textbf{tempered} if for any $i$, all the coordinate
components of $s|_{X_{i}}$ are tempered functions. The space of global
tempered sections of $E$ is denoted by $\mathcal{T}\left(X,E\right)$.
\end{defn}

\begin{rem}
As tempered functions on QN varieties form sheaves, the definition
above does not depend on the cover.
\end{rem}

\begin{defn}
\textit{A Schwartz section on a QN bundle $X$:}\textcolor{black}{{}
let $X$ be a QN variety, and let} $E$ be a QN bundle over $X$.
Let $\bigcup\limits _{i=1}^{m}X_{i}=X$ be affine QN trivialization
of $E$. Denote by $Func\left(X,E,\mathbb{R}\right)$\textcolor{black}{{}
the bundle of all real valued sections of $E$ over $X$.} There is
a natural map $\psi:\bigoplus\limits _{i=1}^{m}\mathcal{S}\left(X_{i}\right)^{n}\to Func\left(X,E,\mathbb{R}\right)$.
Define \textbf{the space of global Schwartz sections of }$E$ by $\mathcal{S}\left(X,E\right):=Im\psi$.
\end{defn}

\begin{rem}
We define the topology on the space of global Schwartz sections of
$E$ by the quotient topology, i.e. by the isomorphism $\mathcal{S}\left(X,E\right)\cong\bigoplus\limits _{i=1}^{m}\mathcal{S}\left(X_{i}\right)^{n}/Ker\psi$.
\end{rem}

$\:$
\begin{rem}
The definition above does not depend on the cover. See the proof of
Lemma \ref{lem-Sch-cover-indep}.
\end{rem}

Now we will define the local sections:
\begin{defn}
Let $X$ be a QN variety, and let $E$ be a QN bundle over it. We
define the \textbf{cosheaf $\mathcal{S}_{X}^{E}$ of Schwartz sections
of $E$ }by $\mathcal{S}_{X}^{E}\left(U\right):=\mathcal{S}\left(U,E|_{U}\right)$.
We define in a similar way the \textbf{sheaf $\mathcal{T}_{X}^{E}$
of tempered sections of $E$} by $\mathcal{T}_{X}^{E}\left(U\right):=\mathcal{T}\left(U,E|_{U}\right)$.
\end{defn}

\begin{lem}
Let $X$ be a QN variety. For an open semi-algebraic subset $U\subset X$,
$\mathcal{S}_{X}^{E}|_{U}=\mathcal{S}_{U}^{E|_{U}}$, $\mathcal{T}_{X}^{E}|_{U}=\mathcal{T}_{U}^{E|_{U}}$
.
\end{lem}

\begin{proof}
The lemma holds by definitions.
\end{proof}
\begin{thm}
\label{thm-bundle-char}(Characterization of Schwartz sections on
open subset - the bundle case) Let $X$ be a QN variety, and let $Z\subset X$
be some closed subset. Define $U:=X\backslash Z$ and $W_{Z}:=\left\{ \phi\in\mathcal{S}\left(X,E\right)|\phi\text{ vanish with all its derivatives on }Z\right\} $.
Then extension by zero $Ext_{U}^{X}:\mathcal{S}_{X}^{E}\left(U\right)\to W_{Z}$
is an isomorphism of Fréchet spaces, whose inverse is $Res_{U}^{X}:W_{Z}\to\mathcal{S}_{X}^{E}\left(U\right)$.
\end{thm}

\begin{proof}
This theorem follows from Theorem \ref{thm-general-char} - characterization
of Schwartz functions on open subset for general QN variety, and Proposition
\ref{prop-part-of-unity} - tempered partition of unity for general
QN variety.
\end{proof}

\appendix

\section{Preperations For Partition Of Unity}

In order to prove partition of unity we will follow the proof of Theorem
\ref{AG-thm-part-of-unity} using some definitions and lemmas, lightly
adapted to our case. We will start with some definitions, and then
show there is a certain refinement for the cover we will need later
on. 
\begin{defn}
\label{def-Basicness}1) Let $M$ be a QN variety and $F$ be a continuous
semi-algebraic function on $M$. We denote $M_{F}:=\left\{ x\in M|F\left(x\right)\neq0\right\} $.

2) Let $M$ be a QN variety. A continuous semi-algebraic function
$F$ on $M$ is called \textbf{basic} if $F|_{M_{F}}$ is a positive
QN function.

3) A collection of continuous semi-algebraic functions $\left\{ F_{i}\right\} $
is called \textbf{basic collection} if every one of them is basic,
and in every point of $M$ one of them is larger than 1.
\end{defn}

\begin{thm}
({[}S III.1.1{]}) Let $r<\infty$. A $C^{r}$ Nash manifold is affine.
Thus, by {[}AG - A.2.4{]}, a QN variety M can be continuously embedded
in $\mathbb{R}^{n}$ by a semi-algebraic map, where M and its image
are homeomorphic.
\end{thm}

\begin{cor}
\label{cor-metric}Let M be a QN manifold. Then there exists a semi-algebraic
continuous metric $d:M\times M\rightarrow\mathbb{R}$.
\end{cor}

\begin{defn}
A cover $M=\bigcup\limits _{j=1}^{m}V_{j}$ is called a \textbf{proper
refinement} of the cover $M=\bigcup\limits _{i=1}^{n}U_{i}$ if for
any $j$ there exists $i$ such that $\bar{V_{j}}\subset U_{i}$.
\end{defn}

\begin{prop}
\label{prop-proper-ref}Let $M=\bigcup\limits _{i=1}^{n}U_{i}$ be
a finite open (semi-algebraic) cover of an affine QN variety M. Then
there exists a finite open (semi-algebraic) cover $M=\bigcup\limits _{j=1}^{m}V_{j}$
which is a proper refinement of $\left\{ U_{i}\right\} $.
\end{prop}

\begin{proof}
Let $d$ be the metric from corollary \ref{cor-metric}. If a set
$A$ is closed in the classical topology, then the distance $d\left(x,A\right):=\inf\limits _{y\in A}d\left(x,y\right)$
is strictly positive for all points $x$ outside $A$. Now define
$F_{i}:M\rightarrow\mathbb{R}$ by $F_{i}\left(x\right)=d\left(x,M\backslash U_{i}\right)$.
It is semi-algebraic by the Tarski-Seidenberg principle. Define $G=\left(\sum\limits _{i=1}^{n}F_{i}\right)/2n$
and $V_{i}=\left\{ x\in M|F_{i}\left(x\right)>G\left(x\right)\right\} $.
It is easy to see that $V_{i}$ is a proper refinement of $U_{i}$.
\end{proof}
\begin{thm}
(finiteness) Let $X\subset\mathbb{R}^{n}$ be a semi-algebraic set.
Then every open semi-algebraic subset of $X$ can be presented as
a finite union of sets of the form 
\[
\left\{ x\in X|p_{i}\left(x\right)>0,\:i=1,...,n\right\} ,
\]
where $p_{i}$ are polynomials in n variables.
\end{thm}

\begin{cor}
\label{cor-basis-Mf}Let M be an \textup{affine} QN variety. Then
it has a basis of open sets of the form $M_{F}$ where F is a basic
function.
\end{cor}

\begin{lem}
\label{lem-majoration} ({[}AG - Lemma A.2.1 + Lemma 2.2.11{]}) Let
M be an \textup{affine} QN variety. Then any continuous semi-algebraic
function on it can be majorated by a QN function, and any continuous
strictly positive semi-algebraic function on it can be bounded from
below by a strictly positive QN function.
\end{lem}

\end{document}